\documentclass[12pt,reqno]{article}

\usepackage[usenames]{color}
\usepackage{amssymb}
\usepackage{amsmath}
\usepackage{amsthm}
\usepackage{amsfonts}
\usepackage{amscd}
\usepackage{graphicx}
\usepackage{mathtools}

\usepackage[colorlinks=true,
linkcolor=webgreen,
filecolor=webbrown,
citecolor=webgreen]{hyperref}

\definecolor{webgreen}{rgb}{0,.5,0}
\definecolor{webbrown}{rgb}{.6,0,0}

\usepackage{color}
\usepackage{fullpage}
\usepackage{float}

\usepackage{graphics}
\usepackage{latexsym}
\usepackage{epsf}
\usepackage{breakurl}

\newcommand{\seqnum}[1]{\href{https://oeis.org/#1}{\rm \underline{#1}}}

\DeclareMathOperator{\dar}{dar}
\DeclareMathOperator{\wt}{wt}
\DeclareMathOperator{\red}{red}
\DeclareMathOperator{\sech}{sech}

\begin{document}


\hfill
\vskip 0.5in

\theoremstyle{plain}
\newtheorem{theorem}{Theorem}
\newtheorem{corollary}[theorem]{Corollary}
\newtheorem{lemma}[theorem]{Lemma}
\newtheorem{proposition}[theorem]{Proposition}

\theoremstyle{definition}
\newtheorem{definition}[theorem]{Definition}
\newtheorem{example}[theorem]{Example}
\newtheorem{conjecture}[theorem]{Conjecture}

\theoremstyle{remark}
\newtheorem{remark}[theorem]{Remark}

\begin{center}
\vskip 1cm{\LARGE\bf 
Pseudo-Involutions in the Riordan Group
}
\vskip 1cm
\large
Alexander Burstein and Louis W. Shapiro\\
Department of Mathematics\\
Howard University\\
Washington, DC 20059\\
USA\\
\href{mailto:aburstein@howard.edu}{\texttt{aburstein@howard.edu}}\\
\href{mailto:lshapiro@howard.edu}{\texttt{lshapiro@howard.edu}}\\
\end{center}

\vskip .2 in

\begin{abstract}
We consider pseudo-involutions in the Riordan group where the generating function $g$ for the first column of a Riordan array satisfies a palindromic or near-palindromic functional equation. For those types of equations, we find, for very little work, the pseudo-involutory companion of $g$ and have a pseudo-involution in a $k$-Bell subgroup. There are only slight differences in the ordinary and exponential cases. In many cases, we also develop a general method for finding $B$-functions of Riordan pseudo-involutions in $k$-Bell subgroups, and show that these $B$-functions involve Chebyshev polynomials. We apply our method for many families of Riordan arrays, both new and already known.

We also have some duality and reciprocity results. Since many of the examples we discuss have combinatorial significance, we conclude with a few remarks on the general framework for a combinatorial interpretation of some of the generating function results we obtain.
\end{abstract}

\section{Preface} \label{sec:preface}

This paper is concerned with the Riordan group, first defined by Shapiro et al.\ \cite{SGWW}, a group of infinite lower triangular matrices with properties generalizing those of the Pascal triangle. Before starting systematically in Section \ref{sec:intro}, we will assume some familiarity with the Riordan group and pseudo-involutions (see \cite{Barry, Zeleke} for an introductory book and lecture) and make a few remarks to add some perspective.

Riordan arrays are useful in enumerating unlabeled or labeled combinatorial structures consisting of a sequence of nonempty structures of the same type connected to a root structure, possibly empty and possibly of a different type. Examples of unlabeled structures include compositions, lattice paths, rooted unlabeled trees, secondary RNA structures, and directed animals, while examples of labeled structures include sets, set partitions, necklaces of set partitions, and rooted labeled trees. As we mentioned before, the basic example of a Riordan array is the Pascal triangle matrix, given by $\left(\frac{1}{1-z},\frac{z}{1-z}\right)$ as an ordinary Riordan array, and by $[e^z,z]$ as an exponential Riordan array. Other well-known examples of ordinary Riordan arrays include the RNA matrix $(R,zR)$ (Example \ref{ex:rna}) and the directed animal matrix $(\tilde{m}(z),z\tilde{m}(z))$ (Example \ref{ex:motzkin-tree}). Well-known examples of exponential Riordan arrays are triangles of the Stirling numbers of the first and second kind.

Desirable properties of Riordan arrays include not only combinatorial meaning, but also recognizable $A$- and $Z$-sequences \cite{HSp,MRSV,R}, or $B$-sequences for pseudo-involutions \cite{B2,B1,CN,CJKS,CK,CKS07,HSh2}. Those sequences let us construct Riordan arrays along their rows or antidiagonals, respectively, since the definition of Riordan arrays implies a column-by-column construction.

In this paper, we are primarily interested in Riordan arrays with certain special properties and other arrays closely related to those. One such property is that the Riordan array be an element of finite order in the Riordan group. It is straightforward to see that the only real-valued non-identity Riordan arrays of finite order are elements of order $2$, i.e., involutions. It turns out, however, that the closely related \emph{pseudo-involutions} are often more ``combinatorial'' as they are more likely to contain only nonnegative entries. The second property of interest is that the Riordan array be contained in a \emph{$k$-Bell subgroup}, i.e., be of the form $(g,f)$ or $[g,f]$ with $f(z)=zg(z)^k$ for some integer $k$. Some well-known examples of pseudo-involutions in $k$-Bell subgroups for various $k$ include the Pascal triangle, the RNA matrix, and the directed animal matrix mentioned above, as well as arrays $(C,zC^3)$ and $(r,zr^2)$, where $C=C(z)$ and $r=r(z)$ are the generating functions for the Catalan and large Schr\"{o}der numbers, respectively. 

In this article we greatly increase the number of families of ``desirable'' pseudo-involutions. We begin in Section \ref{sec:intro} by defining a Riordan array, ordinary and exponential, and its associated $A$- and $Z$-sequence and function, as well the $B$-sequence and $B$-function if the array is a pseudo-involution. This includes the interpretation of the $B$-sequence of an exponential Riordan array in terms of a recursive relation on a related array that we call the associated \emph{reduced} exponential Riordan array.

In Section \ref{sec:pal-bell}, we show that, for functions $g$ of Riordan arrays $(g,f)$ or $[g,f]$, defined by certain functional relations common in combinatorial contexts (e.g., $C=1+zC^2$), we can easily produce the pseudo-involutory companion $f$. Furthermore, we show that for many classic combinatorial examples, where those functional equations involve palindromic polynomials, the resulting Riordan pseudo-involution is in the $k$-Bell subgroup for some $k$. 

In Section \ref{sec:b-seq}, we develop a general method for finding $B$-functions of $k$-Bell Riordan pseudo-involutions and apply it to determine $B$-functions for several families of Riordan arrays, both new and already known. An interesting feature of those $B$-functions is that they all involve Chebyshev polynomials. We also find the relation between the $B$-function of a $k$-Bell Riordan pseudo-involution and the $B$-function of its aerated version in the $1$-Bell subgroup (simply called the \emph{Bell subgroup}).

For example, consider walks from $(0,0) $ to $(3n,0)$ on or above the $x$-axis with steps $(2,1)$, $(1,2)$ and $(1,-1)$, a class also considered in Drake \cite[Example 1.6.9]{Drake}. It can be shown that the ordinary generating function $r_3=r_3(z)$ for the counting sequence \seqnum{A027307} of such walks satisfies the recurrence relation $r_3=1+z(r_3^2+r_3^3)$, from which we can infer that $(r_3, zr_3^4)$ is a pseudo-involution that begins 
\[
\begin{bmatrix}
1 & 0 & 0 & 0 & 0 \\ 
2 & 1 & 0 & 0 & 0 \\ 
10 & 10 & 1 & 0 & 0 \\ 
66 & 90 & 18 & 1 & 0 \\ 
498 & 810 & 234 & 26 & 1
\end{bmatrix}.
\]
However, to find the $B$-function for $zr_3^4$, it is better to consider $r_3$ not on its own but as part of a family of functions $r_k=r_k(z)$, $k\ge 1$, that satisfy the equation $r_k=1+z(r_k^{k-1}+r_k^k)$. This family also includes the generating function $r=r(z)$ for the large Schr\"{o}der numbers as $r=r_2$. The functional relation for $r_k$ lets us show that the Riordan array $(r_k,zr_k^{2k-2})$ is a pseudo-involution in the $(2k-2)$-Bell subgroup, and its $B$-function is $\frac{4}{1-z}U_{k-2}\left(\frac{1+z}{1-z}\right)$, where $U_n$ is the $n$-th Chebyshev polynomial of the second kind. Thus, we obtain $B_{zr_3^4}=\frac{8(1+z)}{(1-z)^2}$. 

In Section \ref{sec:b-pal}, we extend our results from the previous sections to new pseudo-involutions where the pseudo-involutory companion $f$ in $(g,f)$ or $[g,f]$ is more difficult to find given the function $g$. For example, we show that the pseudo-involutory companion of the Fibonacci generating function $F=F(z)=\frac{1}{1-z-z^2}$ is $(F-1)C(F-1)=(zC)\circ(F-1)$, where $C=C(z)$ is the Catalan generating function, and the array $(F,(F-1)C(F-1))$ begins
\[
\begin{bmatrix}
1 & 0 & 0 & 0 & 0 & 0 \\ 
1 & 1 & 0 & 0 & 0 & 0 \\ 
2 & 4 & 1 & 0 & 0 & 0 \\ 
3 & 14 & 7 & 1 & 0 & 0 \\ 
5 & 50 & 35 & 10 & 1 & 0 \\ 
8 & 190 & 160 & 65 & 13 & 1
\end{bmatrix}
\]
However, developing a general method that would simplify finding the $B$-function of this array remains an open problem. Nevertheless, we find by computing a few initial terms, comparing with existing OEIS sequences, and checking by substitution, that
\[
B_{(F-1)C(F-1)}(z)=\frac{1+z-\sqrt{1-10z+5z^2}}{2z}.
\]
Another example of such a function $g$ is the generating function $m=m(z)$ for the Motzkin numbers with the pseudo-involutory companion $f=\frac{m-\sqrt{4m-3m^2}}{2}$. We show (see Example \ref{ex:hfh_fm}) how the pseudo-involution $\left(m, \frac{m-\sqrt{4m-3m^2}}{2}\right)$ can be easily recovered from the Pascal triangle using an operation we call \emph{pseudo-conjugation}. As in the previous example, a general approach to finding the $B$-function under such an operation remains an open problem.

Among the new pseudo-involutions we find is a family that includes partial noncrossing matchings and some types of secondary RNA structures. Another new family contains the pseudo-involution $(1/C(-zC^2), z/C(-zC^2))$ in the Bell subgroup, where $1/C(-zC^2)$ is the generating function for the basketball walks \cite{AZ,BKKKKNW,BFMR}. Yet another new pseudo-involution we find in a closely related family is $(C(zC), zC^5(zC))$ in the 5-Bell subgroup, where $C(zC)$ is the generating function for the number of permutations of length $3n$, $n\ge 0$, that avoid pattern $132$ and decompose into $n$ disjoint 3-cycles $(a,b,c)$ with $a>b>c$ (see Archer and Graves \cite{AG}). Alternatively, $C(zC)$ is the generating function for the number of permutations that avoid the set of patterns $\{1342,3142,45132\}$ and do not start with a descent.

Finally, we conclude with a few remarks on the general framework for a combinatorial interpretation of some of the generating function results we have obtained.

\section{Introduction} \label{sec:intro}

Consider the set $\mathcal{F}=\mathbb{R}[[z]]$ of formal power series with real coefficients. We define the \emph{order} of $f(z)\in \mathbb{R}[[z]]$, where $f=\sum_{n=0}^{\infty}a_nz^n$, as the smallest index $n$ such that $a_n\ne 0$, and let $\mathcal{F}_r$ be the set of all formal power series of order $r$. Whenever there is no confusion, we will omit the variable (which we will assume to be $z$) and simply write $f$ instead of $f(z)$. (Note that this setup readily generalizes to any field $F$ in place of $\mathbb{R}$, but we restrict ourselves to $\mathbb{R}$ since our interest is motivated by possible combinatorial interpretations of the coefficients of the formal power series $f$.)

A lower triangular matrix $D=[d_{ij}]_{i,j\ge 0}$ is called a \emph{Riordan array} (or sometimes, an \emph{ordinary Riordan array}) if there exist $g\in\mathcal{F}_0$ and $f\in\mathcal{F}_1$ such that, for any $k\ge 0$, the generating function $\sum_{n=0}^{\infty}{d_{nk}z^n}$ of the $k$th column of $D$ is $gf^k$. We will denote this by writing $D=(g,f)$.

Given a column vector $\mathbf{h}=[h_k]_{k\ge 0}^T$ and its generating function $h(z)=\sum_{k=0}^{\infty}h_kz^k$, the generating function of the product $D\mathbf{h}$ is given by
\begin{equation} \label{eq:ftra}
(g,f)*h=\sum_{k=0}^{\infty}h_k gf^k = g\sum_{k=0}^{\infty}h_k f^k = g\cdot (h\circ f)=gh(f).
\end{equation}
Equation \eqref{eq:ftra} is called the \emph{Fundamental Theorem of Riordan Arrays} \cite{SGWW} (abbreviated FTRA). From this identity, it is straightforward to see that the matrix multiplication induces the following multiplication law for Riordan arrays:
\begin{equation} \label{eq:mult-ra}
(g,f)*(G,F)=(g\cdot (G\circ f),F\circ f)=(gG(f),F(f)).
\end{equation}
Given a function $f\in\mathcal{F}_1$, let $\bar{f}\in\mathcal{F}_1$ denote its compositional inverse. Then we see that the set of Riordan arrays is closed under matrix multiplication, with $(1,z)$, the identity matrix, as the identity element and
\begin{equation} \label{eq:inv-ra}
(g,f)^{-1}=\left(\frac{1}{g\circ \bar{f}}\ ,\bar{f}\right)=\left(\frac{1}{g(\bar{f})}\ ,\bar{f}\right)
\end{equation}
for any Riordan array $(g,f)$. This means that the Riordan arrays form a group, called the \emph{Riordan group} \cite{SGWW}.

Given any Riordan array $(g,f)$, we can write 
\[
(g,f)=(G,F)*(g(0),|f'(0)|z)
\]
for some $G\in\mathcal{F}_0$, $F\in\mathcal{F}_1$, such that $G(0)=1$ and $F'(0)=\pm1$. Note that if both $G,F\in\mathbb{Z}[[z]]$, then both $1/G, \bar{F}\in\mathbb{Z}[[z]]$ as well, and moreover, $(1/G)(0)=1$ and $\bar{F}'(0)=\pm1$. Since our motivation for considering Riordan arrays is mostly combinatorial, we will assume from now on, with little loss of generality, that Riordan arrays $(g,f)$ have $g(0)=1$ and $f'(0)=\pm1$.

For example, consider the Pascal triangle as a lower triangular matrix $P=\left[\binom{i}{j}\right]_{i,j\ge 0}$. The generating function of the $k$th column of $P$ is
\[
\sum_{i=0}^{\infty}{\binom{i}{k}z^i}=\frac{z^k}{(1-z)^{k+1}}=\frac{1}{1-z}\left(\frac{z}{1-z}\right)^k,
\]
so
\begin{equation} \label{eq:pascal}
P=\left(\frac{1}{1-z}\ ,\ \frac{z}{1-z}\right).
\end{equation}

\bigskip

The Riordan group contains several subgroups whose elements often occur in combinatorial contexts. Those include:
\begin{itemize}

\item the \emph{Lagrange} subgroup of Riordan arrays of the form $(1,f)$;

\item the \emph{Appell} subgroup of Riordan arrays of the form $(g,z)$, which is normal in the Riordan group;

\item the \emph{Bell} subgroup of Riordan arrays of the form $(g,zg)$;

\item the \emph{$k$-Bell} subgroup, for each $k\ge 0$, of Riordan arrays of the form $(g,zg^k)$, which yields the \emph{Appell} subgroup for $k=0$ and the \emph{Bell} subgroup for $k=1$;

\item the \emph{checkerboard} subgroup of Riordan arrays of the form $(g,f)$, where $g$ is even and $f$ is odd, that is $(g,f)=(u(z^2),zv(z^2))$ for some $u,v\in\mathcal{F}_0$;

\item more generally, for each $k\ge 1$, the \emph{$k$-checkerboard} subgroup of Riordan arrays of the form $(u(z^k),zv(z^k))$, where $u,v\in\mathcal{F}_0$;

\item the \emph{derivative} subgroup of Riordan arrays of the form $(f',f)$ (where the multiplication law is just the chain rule);

\item the \emph{hitting time} subgroup of Riordan arrays of the form $(zf'/f,f)$ (these are just derivative subgroup arrays propagated one column to the left).

\end{itemize}

Note also that for any Riordan array $(g,f)$, we have
\[
(g,f)=(g,z)*(1,f)=(1,f)*(g(\bar{f}),z),
\]
so that the Riordan group is a semidirect product of its Appell and Lagrange subgroups.

\bigskip

Considering the rows of Riordan arrays, Rogers \cite{R} and Merlini et al.\ \cite{MRSV} defined the $A$-sequence and $Z$-sequence, respectively, of a Riordan array $(g,f)$, whose generating functions $A_f$ and $Z_{(g,f)}$ satisfy
\begin{equation} \label{eq:az-rel}
f=z(A_f\circ f), \qquad g=\frac{g(0)}{1-z(Z_{(g,f)}\circ f)}
\end{equation}
or, equivalently,
\begin{equation} \label{eq:az-def}
A_f=\frac{z}{\bar{f}}\ , \qquad Z_{(g,f)}=\frac{g-g(0)}{zg}\circ \bar{f}.
\end{equation}
Translating this into relations between the matrix entries of $(g,f)=[d_{ij}]_{i,j\ge 0}$, and the coefficients of $A_f(z)=\sum_{j=0}^{\infty}{a_j z^j}$ and $Z_{(g,f)}(z)=\sum_{j=0}^{\infty}{\zeta_j z^j}$, we obtain
\begin{equation} \label{eq:az-mat}
\begin{split}
d_{n+1,k+1}&=\sum_{j=0}^{\infty}{a_j d_{n,k+j}}, \quad n,k\ge 0,\\
d_{n+1,0}&=\sum_{j=0}^{\infty}{\zeta_j d_{n,j}}, \quad n\ge 0.
\end{split}
\end{equation}
For more on properties of $A$- and $Z$-sequences of Riordan arrays, see, e.g., He and Sprugnoli \cite{HSp}.

In this paper, our particular interest will be in the Riordan group elements of finite order. With $\mathbb{R}$ as the underlying coefficient field, the only formal power series of finite order with respect to composition are involutions, and thus, the only non-identity Riordan arrays of finite order are those of order 2. It is a straightforward exercise to show that if $(g,f)$ is an involution, i.e., $(g,f)^2=(1,z)$, then either $f'(0)=-1$ or $f=z$, and in the latter case $(g,f)=(\pm1,z)$. Thus, non-identity involutions with $g(0)=1$ always have $f'(0)=-1$. Note also that $(g,f)^2=(1,z)$ is equivalent to
\begin{equation} \label{eq:pseudo-gf}
g\cdot(g\circ f)=1, \quad f\circ f=z.
\end{equation}

A Riordan array $(g,f)$ is called a \emph{pseudo-involution} if $(g,-f)$ is an involution. Equivalently, $(g,f)=[d_{ij}]_{i,j\ge 0}$ is a pseudo-involution if and only if $(g,f)^{-1}=[(-1)^{i-j}d_{ij}]_{i,j\ge 0}$. Since $(g,-f)=(g,f)*(1,-z)$, it follows that $(g,f)$ is a pseudo-involution if and and only if
\[
(g,f)^{-1}=(1,-z)*(g,f)*(1,-z)=(g(-z),-f(-z)),
\]
or, equivalently,
\[
g\circ(-f)=\frac{1}{g} \quad \text{and} \quad \bar{f}=-f(-z)=(-z)\circ f \circ (-z).
\]

If $(g,f) $ is a pseudo-involution, we say that $g$ and $f$ are \emph{companions}. A given $g$ can have at most one companion $f$, but a given $f$ has infinitely many companions $g$. In particular, it is easy to check \cite{PS} that $(g^k,f)$ is also a pseudo-involution for any power $k$. When $\bar{f}=-f(-z)$, we will call $f$ \emph{pseudo-involutory} (so as not use the same term for a function and a Riordan array).

If $(g,f)$ is a pseudo-involution, then $zf\in\mathcal{F}_2$ and its coefficient at $z^2$ is 1, thus $\sqrt{zf}\in\mathcal{F}_1$. Furthermore, it is easy to check that $(zf)\circ(-f)=(-f)\cdot(-z)=zf$, and thus, $zf$ is a fixed point of $-f$, which implies that if $f\ne -z$ then $\sqrt{zf}\circ(-f)=-\sqrt{zf}$. Moreover, $f-z\in\mathcal{F}_2$ and $(f-z)\circ(-f)=-z-(-f)=f-z$, i.e., $f-z$ is also a fixed point of $-f$. Therefore, for a pseudo-involution $(g,f)$ we can define a so-called \emph{$B$-function} of $f$, $B_f=B_f(z)$, such that
\begin{equation} \label{eq:b-seq-def}
f-z=(zB_f)\circ(zf)=zfB_f(zf).
\end{equation} 
If $B_f(z)=\sum_{j=0}^{\infty}{b_j z^j}$, then this is equivalent to the fact that
\begin{equation} \label{eq:b-mat}
d_{n+1,k+1}=d_{n,k}+\sum_{j=0}^{\infty}{b_j d_{n-j,k+j+1}}, \quad n,k\ge 0,
\end{equation}
and $\{b_n\}_{n\ge 0}$ is called the \emph{$B$-sequence} of $f$. The $B$-sequence and $B$-function of a Riordan array were first introduced in Cheon et al.\ \cite{CJKS} as the $\Delta$-sequence and $\Delta$-function, respectively, but were subsequently called the $B$-sequence and $B$-function in later papers \cite{He1, HSh2,PS}.

Pseudo-involutions often have combinatorial interpretations as distributions of combinatorial statistics on combinatorial classes. For example, it is easy to check that Pascal triangle \eqref{eq:pascal} is a pseudo-involution. Furthermore, for the Pascal triangle $P$, we have
\[
f-z=\frac{z}{1-z}-z=\frac{z^2}{1-z}=zf,
\]
so that $B_f=1$. In this case, both the $A$- and $B$-sequences simply yield the binomial formula $\binom{n+1}{k+1}=\binom{n}{k}+\binom{n}{k+1}$. 

\subsection{Exponential Riordan group} \label{subsec:exp_riordan}

The preceding discussion assumed the use of ordinary generating functions. Some adjustments are needed in order to use the exponential generating functions as companions in a Riordan pair. Specifically, a lower triangular matrix $M=[a_{n,k}]_{n,k\ge 0}$ is an \emph{exponential Riordan array} $[g,f]$ \cite{DS04, B07}, for exponential generating functions $g\in\mathcal{F}_0$ and $f\in\mathcal{F}_1$, if the $k$th column of $M$ (starting on the left with $k=0$) is given by $\frac{gf^k}{k!}$. It follows that
\begin{equation} \label{eq:exp-def}
a_{n,k}=\left[\frac{z^n}{n!}\right]\frac{gf^k}{k!}=\frac{n!}{k!}[z^n](gf^k)=\frac{n!}{k!}d_{n,k}\,, \qquad 0\le k\le n,
\end{equation}
where $[d_{n,k}]_{n,k\ge 0}=(g,f)$. The FTRA \eqref{eq:ftra} still holds for the exponential Riordan arrays, and the identity matrix is again given by $[1,z]$.

For example, the Pascal triangle is the exponential Riordan array $P=[e^z,z]$ since
\[
a_{n,k}=\frac{n!}{k!}[z^n](z^ke^z)=\frac{n!}{k!}[z^{n-k}]e^z=\frac{n!}{k!(n-k)!}=\binom{n}{k}.
\]
Similarly, for any fixed integer $k\ge 0$, the exponential generating function for the Stirling numbers of the second kind $S(n,k)$ is $\frac{(e^z-1)^k}{k!}$, so the matrix $[S(n,k)]_{n,k\ge 0}$ is the exponential Riordan array $[1,e^z-1]$. Likewise, letting $s(n,k)$, $0\le k\le n$, be the signed Stirling numbers of the first kind, we have
\[
[s(n,k)]_{n,k\ge 0}=[S(n,k)]_{n,k\ge 0}^{-1}=[1,e^z-1]^{-1}=[1,\overline{(e^z-1)}]=[1,\log(1+z)].
\]
Therefore, for the signless Stirling numbers of the first kind, $c(n,k)=(-1)^{n-k}s(n,k)$, we obtain
\[
[c(n,k)]_{n,k\ge 0}=[1,-z]*[1,\log(1+z)]*[1,-z]=[1,-\log(1-z)]=\left[1,\log\left(\frac{1}{1-z}\right)\right].
\]

The involutions and pseudo-involutions in the exponential Riordan group are defined just like those in the Riordan group: $[g,f]$ is an involution if $[g,f]^2=[1,z]$, and $[g,f]$ is a pseudo-involution if $[g,-f]$ is an involution.

The definition of the $B$-sequence for the exponential Riordan array also requires some adjustments so as to preserve the functional relation \eqref{eq:b-seq-def} for the $B$-function. Specifically, with $a_{n,k}$ and $b_j$ defined as in \eqref{eq:exp-def} and \eqref{eq:b-mat}, respectively, let 
\begin{equation} \label{eq:def-alpha-beta}
\alpha_{n,k}=\frac{a_{n,k}}{\binom{n}{k}}, \quad n\ge k\ge 0; \qquad \quad \alpha_{n,k}=0, \quad n<k; \qquad \quad 
\beta_j=(2j+1)!\,b_j\, , \quad j\ge 0;
\end{equation}
so that $\alpha_{n,k}$ ($n\ge k\ge 0$) is the coefficient at $\frac{z^{n-k}}{(n-k)!}$ of the power series expansion of $g\cdot(f/z)^k$, with $f/z\in\mathcal{F}_0$, and
\begin{equation} \label{eq:b-exp-def}
B_f(z)=\sum_{j=0}^{\infty}b_jz^j=\sum_{j=0}^{\infty}\beta_j\frac{z^j}{(2j+1)!}.
\end{equation}
In other words, $B_f(z)=\frac{Y_f\left(\sqrt{z}\right)}{\sqrt{z}}$ for some odd exponential generating function $Y_f(z)$. Then
\begin{equation} \label{eq:b-seq-exp-def}
\alpha_{n+1,k+1}=\alpha_{n,k}+\sum_{j=0}^{\infty}{\binom{n-k}{2j+1}\beta_j \alpha_{n-j,k+j+1}}, \quad n\ge k\ge 0.
\end{equation}

For example, the Pascal triangle as the exponential Riordan array $P=[e^z,z]$ has $f=z$, so $zfB_f(zf)=f-z=0$ and thus $B_f=0$. Therefore, we have $\alpha_{n,k}=\frac{\binom{n}{k}}{\binom{n}{k}}=1$ for all $n\ge k\ge 0$ and $\beta_j=0$ for all $j\ge 0$, and thus we see that Equation \eqref{eq:b-seq-exp-def} holds for $P$.

\begin{definition} \label{def:exp-red}
Given an exponential Riordan array $[g,f]=[a_{n,k}]_{n,k\ge 0}$, define its \emph{reduction} by $\red[g,f]=[\alpha_{n,k}]_{n,k\ge 0}$, where $\alpha_{n,k}$ is as in \eqref{eq:def-alpha-beta}. We call $\red[g,f]$ the \emph{reduced} exponential array associated with $[g,f]$, and call the sequence $\{\beta_j\}_{j=0}^{\infty}$ in \eqref{eq:def-alpha-beta} the \emph{beta-sequence} of $[g,f]$.
\end{definition}

Since $\alpha_{n,0}=a_{n,0}$, the exponential generating function for the leftmost column of $\red[g,f]$ is simply $g$. Knowing the first column, we can use the sequence $\{\beta_j\}_{j=0}^{\infty}$ together with \eqref{eq:b-seq-exp-def} to construct the rest of $\red[g,f]$ recursively. Finally, given any $\alpha_{n,k}$ in $\red[g,f]$, we find $a_{n,k}=\binom{n}{k}\alpha_{n,k}$ to construct the exponential Riordan array $[g,f]$.

\section{Generalized palindromes and the $k$-Bell subgroup} \label{sec:pal-bell}

In this section, we will consider two special cases of a recursive formula for $g$, the first component of a Riordan array $(g,f)$, that imply that $f=zg^k$ for some $k\in\mathbb{Z}$. In other words, all these Riordan arrays are pseudo-involutions in the $k$-Bell subgroup.

Consider a polynomial $p(z)=\sum_{j=\ell_*}^{\ell^*}{a_jz^j}$, where $a_{\ell_*},a_{\ell^*}\ne0$. Then the \emph{reciprocal} polynomial of $p(z)$ is
\[
p_{\mathrm{rec}}(z)=\sum_{j=\ell_*}^{\ell^*}{a_{\ell_*+\ell^*-j}z^j}=z^{\ell_*+\ell^*}p\left(\frac{1}{z}\right).
\]
Moreover, a polynomial $p$ is called \emph{palindromic}, or a \emph{palindrome}, if $p_{\mathrm{rec}}=p$. 

Finally, we call a function $\gamma=\gamma(z)$ a \emph{generalized palindrome} if
\[
\frac{\gamma(z)}{\gamma\left(\frac{1}{z}\right)}=z^k
\]
for some $k\in\mathbb{Q}$. In this case, we call $k$ the \emph{darga} of $\gamma$ (stress on the second syllable), denoted $k=\dar(\gamma)$. For example, for a palindromic polynomial $p$ as above, we have $\dar(p)=\ell_*+\ell^*$.

Let $\mathcal{P}_d$ be the set of generalized palindromes of darga $d$, and let $\mathcal{P}=\bigcup_{d\in\mathbb{Q}}\mathcal{P}_d$. It is routine to check that generalized palindromes satisfy the following properties: if $\gamma_1=\gamma_1(z)\in\mathcal{P}_{d_1}$ and $\gamma_2=\gamma_2(z)\in\mathcal{P}_{d_2}$, then

\begin{itemize}

\item $\gamma_1\gamma_2\in\mathcal{P}_{d_1+d_2}$,

\item $\gamma_1/\gamma_2\in\mathcal{P}_{d_1-d_2}$,

\item $\gamma_1^r\in\mathcal{P}_{rd_1}$ for $r\in\mathbb{Q}$,

\item $\gamma_1(z^r)\in\mathcal{P}_{rd_1}$ for $r\in\mathbb{Q}$,

\item if $d_1=d_2=d$ and $\gamma_1+\gamma_2\ne 0$, then $\gamma_1+\gamma_2\in\mathcal{P}_d$,

\item if $d_1=0$, then $\eta\circ\gamma_1\in\mathcal{P}_0$ for any $\eta=\eta(z)\ne 0$.

\end{itemize}

Most of the results in this paper stem from two simple observations, one more often useful in the case of ordinary generating functions, and the other more often useful in the case of exponential generating functions.

\begin{theorem} \label{thm:palindrome} \hfill \\[-\baselineskip]
\begin{enumerate}
\item Let $g=g(z)$ be a function that satisfies a functional relation 
\begin{equation} \label{eq:gamma-ogf}
g=1+z\gamma(g)=1+z\cdot(\gamma\circ g),
\end{equation}
for some function $\gamma$. Then $(g,f)$ and $[g,f]$ are pseudo-involutions if and only if
\begin{equation} \label{eq:f-ogf}
f=z\frac{\gamma(g)}{g\gamma\left(\frac{1}{g}\right)}.
\end{equation}
Moreover, if $\gamma\in\mathcal{P}_d$, then $f=zg^{d-1}$, so $(g,f)$ and $[g,f]$ are in the $(d-1)$-Bell subgroup.
\item Let $g=g(z)$ be a function that satisfies a functional relation 
\begin{equation} \label{eq:gamma-egf}
g=e^{z\gamma(g)}=e^{z\cdot(\gamma\circ g)},
\end{equation}
for some function $\gamma$. Then $(g,f)$ and $[g,f]$ are pseudo-involutions if and only if
\begin{equation} \label{eq:f-egf}
f=z\frac{\gamma(g)}{\gamma\left(\frac{1}{g}\right)}.
\end{equation}
Moreover, if $\gamma\in\mathcal{P}_d$, then $f=zg^d$, so $(g,f)$ and $[g,f]$ are in the $d$-Bell subgroup.
\end{enumerate}
\end{theorem}

\begin{proof}
In the first case, we have
\[
z=\frac{g-1}{\gamma(g)}.
\]
Recall that for a pseudo-involution $(g,f)$, we have $g(-f)=1/g$, so
\[
-f=\frac{g(-f)-1}{\gamma(g(-f))}=\frac{\frac{1}{g}-1}{\gamma\left(\frac{1}{g}\right)}=-\frac{g-1}{g\gamma\left(\frac{1}{g}\right)}=-\frac{z\gamma(g)}{g\gamma\left(\frac{1}{g}\right)},
\]
and hence,
\[
f=z\frac{\gamma(g)}{g\gamma\left(\frac{1}{g}\right)}.
\]
In the second case, we have
\[
z=\frac{\log(g)}{\gamma(g)},
\]
so
\[
-f=\frac{\log(g(-f))}{\gamma(g(-f))}=\frac{\log\left(\frac{1}{g}\right)}{\gamma\left(\frac{1}{g}\right)}=-\frac{\log(g)}{\gamma\left(\frac{1}{g}\right)}=-\frac{z\gamma(g)}{\gamma\left(\frac{1}{g}\right)},
\]
and hence,
\[
f=z\frac{\gamma(g)}{\gamma\left(\frac{1}{g}\right)}. \qedhere
\]
\end{proof}

We begin with an example where $\gamma$ is not necessarily a generalized palindrome.

\begin{example} \label{eg:non-pal}
Suppose that $g=g(z)$ satisfies the equation $g=1+z(bg-a)$ for some constants $a,b$ such that $a\ne b$ (if $a=b$, then $g=1$, so $(g,f)$ is a pseudo-involution whenever $f$ is pseudo-involutory). Then $\gamma(z)=bz-a$ and
\[
g=\frac{1-az}{1-bz},
\]
and if $(g,f)$ is a pseudo-involution, then
\[
f=z\,\frac{bg-a}{g\cdot\left(\frac{b}{g}-a\right)}=z\,\frac{bg-a}{b-ag}=\frac{z}{1-(a+b)z}.
\]
As an example, we take $a=1$ and $b=2$ to get
\[
\left(\frac{1-z}{1-2z}\, ,\frac{z}{1-3z}\right)= 
\begin{bmatrix}
1 & 0 & 0 & 0 & 0 & \!\!\cdots \\ 
1 & 1 & 0 & 0 & 0 & \!\!\cdots \\ 
2 & 4 & 1 & 0 & 0 & \!\!\cdots \\ 
4 & 14 & 7 & 1 & 0 & \!\!\cdots \\
8 & 46 & 35 & 10 & 1 & \!\!\cdots\\[-3pt]
\vdots & \vdots & \vdots & \vdots & \vdots & \!\!\ddots
\end{bmatrix}.
\]
Note also that letting $a=0$ and $a=-b$, respectively, leads to well-known examples $\left(\frac{1}{1-bz}\, ,\frac{z}{1-bz}\right)$ in the Bell subgroup and $\left(\frac{1+bz}{1-bz}\, , z\right)$ in the Appell subgroup, respectively.
\end{example}

Any pseudo-involution $(g,zg^k)$ can be expanded into a family of pseudo-involutions as follows.

\begin{corollary} \label{cor:palindrome}
Let $g=g(z)$ be as in Theorem \ref{thm:palindrome}, and let $k=d-1$ (resp., $k=d$) if $g$ satisfies the functional relation \eqref{eq:gamma-ogf} (resp., \eqref{eq:gamma-egf}), where $\gamma$ is a generalized palindrome. Let $q\in\mathbb{N}$, and let $p=k/q$. Then
\[
\left(g(z^q),zg^p\!\left(z^{q}\right)\right)^{-1}=\left(g(-z^q),zg^p\!\left(-z^{q}\right)\right).
\]
In particular, this implies the following:
\begin{description}
\item[Special Case 1:] for $q=k$ and $p=1$, we have $\left(g(z^k),zg\left(z^{k}\right)\right)^{-1}=\left(g(-z^k),zg\left(-z^{k}\right)\right)$,
\item[Special Case 2:] if $q$ is odd, then $\left(g(z^q),zg^p\!\left(z^{q}\right)\right)$ is a pseudo-involution,
\item[Special Case 3:] if $k$ is odd, then $\left(g(z^k),zg(z^k)\right)$ is a pseudo-involution.
\end{description}
\end{corollary}

\begin{proof}
Let $h=h(z)=zg^p\!\left(z^{q}\right)$, then
\[
z^q\circ h = h^q=z^q g^{pq}\!\left(z^{q}\right)=z^q g^{k}\!\left(z^{q}\right)=(zg^k)\circ(z^q).
\]
Moreover,
\[
g(-z^q)\circ h = g\circ (-h^q) = g \circ (-zg^k) \circ z^q= \frac{1}{g}\circ z^q = \frac{1}{g(z^q)},
\]
so
\[
\left(zg^p(-z^q)\right)\circ h = h\cdot (g^p\circ (-h^q)) = zg^p\!\left(z^{q}\right)\cdot \left(\frac{1}{g(z^q)}\right)^p = z,
\]
and therefore
\[
\left(g(z^q),zg^p\!\left(z^{q}\right)\right) \left(g(-z^q),zg^p\!\left(-z^{q}\right)\right) = (1,z).
\]
A similar manipulation shows that $\left(g(-z^q),zg^p\!\left(-z^{q}\right)\right)\left(g(z^q),zg^p\!\left(z^{q}\right)\right) = (1,z)$ as well, and thus $\left(g(z^q),zg^p\!\left(z^{q}\right)\right)^{-1}=\left(g(-z^q),zg^p\!\left(-z^{q}\right)\right)$.

\medskip

When $q$ is odd, we have $z^q\circ (-z)=(-z)^q=-z^q=(-z)\circ z^q$, so
\[
g(z^q)\circ(-h)=g\circ (-h)^q = g\circ (-h^q) = \frac{1}{g(z^q)},
\]
and hence $(g(z^q), h)$ is a pseudo-involution.
\end{proof}

In fact, note that for odd $k$ in Corollary \ref{cor:palindrome}, $(g,zg^k)$ is a pseudo-involution if and only if the array $\left(g(z^k),zg(z^k)\right)$ is also a pseudo-involution, a fact that we will use later in the paper. In this situation, if $k=pq$ for some odd integer $q\in\mathbb{N}$ and $p\in\mathbb{Q}$, then we call $\left(g(z^q),zg^p\!\left(z^{q}\right)\right)$ a \emph{$q$-aeration} of $(g,zg^k)$, and in particular, if $p=1$ and $q=k$, simply an \emph{aeration} of $(g,zg^k)$.

We will now discuss some examples where $\gamma\in\mathcal{P}_d$, in particular, when $\gamma$ is a palindromic polynomial.

\begin{example} \label{ex:pal-tree} \emph{(Palindromic trees)}
Consider the class of vertex-colored rooted ordered trees with color set $\{\kappa_i \mid 0\le \ell_*\le i\le \ell^*\}$ where a vertex $v$ of color $\kappa_i$ has weight $\wt(v)=a_i$ and is the root of a subtree with $i$ ordered, possibly empty subtrees. Then the weight of a tree $T$ is defined as the product of weights of its vertices, $\wt(T)=\prod_{v\in T}{\wt(v)}$. This allows the weights $a_j$ to be negative as well. Then the generating function $g=g(z)$ for the signed weighted enumeration of such trees is
\begin{equation} \label{ex:pal-trees-eq}
g=1+\sum_{i=\ell_*}^{\ell^*}a_izg^i=1+z\gamma(g),
\end{equation}
where $\gamma=\gamma(z)=\sum_{i=\ell_*}^{\ell^*}{a_i z^i}$. We say that such a tree satisfies the \emph{palindromic condition}, or is \emph{palindromic}, if $a_i=a_{d-i}$ for $d=\ell_*+\ell^*$ and $i\in[\ell_*,\ell^*]$. Note that, for a palindromic class of trees as above, the polynomial $\gamma$ is palindromic as well, and $\dar(\gamma)=d$. 

Alternatively, we can dispense with colors and leave only weights as follows. Consider a class of rooted ordered trees where a vertex of degree $i\in[0,\ell^*]$ has weight $\sum_{k=i}^{l^*}\binom{k}{i}a_k$ (and we set $a_i=0$ for $i\notin[\ell_*,\ell^*]$). Note that $\sum_{k=i}^{l^*}\binom{k}{i}a_k=[z^i]\gamma(z+1)$ for $0\le i\le\ell^*$. Let $g$ be the ordinary generating function for the enumerating sequence of this class. Then $G=g-1$ satisfies the equation $G=z\cdot(\gamma(z+1)\circ G)=z\gamma(G+1)$, which is equivalent to \eqref{ex:pal-trees-eq}.
\end{example}

Here are several examples (or families of examples) of palindromic trees.

\begin{example} \label{ex:catalan-tree} \emph{($k$-ary trees)} Consider the class of (incomplete) $k$-ary trees, where $\ell_*=\ell^*=k$ (see Example \ref{ex:pal-tree}) and all vertices have weight $a\ne0$. If $g=g(z)$ is the generating function for this class, then $g=1+azg^k$, so $g=T_k(az)$, where $T_k=T_k(z)$ is the generating function for the uncolored $k$-ary trees, satisfying the equation $T_k=1+zT_k^k$. We also have $\gamma(z)=az^k$, so $\gamma$ is palindromic with $\dar(\gamma)=k+k=2k$, and thus
\[
\left(T_k(az),zT_k^{2k-1}(az)\right)
\]
is a pseudo-involution. For example, this yields pseudo-involutions $(1+z,\frac{z}{1+z})$ when $k=0$, $(\frac{1}{1-z},\frac{z}{1-z})$ when $k=1$, and $(C,zC^3)$ when $k=2$, where $C=C(z)=\frac{1-\sqrt{1-4z}}{2z}$ is the Catalan generating function (see \seqnum{A000108}) and $zC^3$ is the generating function of \seqnum{A000245}.
\end{example}

\begin{example} \label{ex:schroeder-tree} \emph{(Schr\"{o}der trees)}
If $\gamma=\gamma(z)$ is palindromic and has only two nonzero terms, then $\gamma=az^{k}+az^{\ell}$ for some $a\ne 0$, and $\dar(\gamma)=k+\ell$, so
\[
(g,zg^{k+\ell-1})
\]
is a pseudo-involution. For example, setting $k=1$ and $\ell=2$ yields $g=r(az)$ and the pseudo-involution 
\[
\left(r(az),zr^2(az)\right),
\]
where $r=r(z)$ is the generating function for the large Schr\"{o}der numbers \seqnum{A006318} that satisfies the functional equation $r=1+zr+zr^2$. Recall that the $n$-th large Schr\"{o}der number counts the number of \emph{Schr\"{o}der paths}, i.e., lattice paths from $(0,0)$ to $(2n,0)$ with steps $(1,1)$, $(1,-1)$, $(2,0)$ that stay on or above the $x$-axis. The related little Schr\"{o}der numbers \seqnum{A001003}, whose generating function $s=s(z)$ satisfies $s=(r-1)/2=1/(1-zr)$, count the lattice paths with the same steps and no $(2,0)$ (``level'') steps on the $x$-axis.
\end{example}

\begin{example} \label{ex:motzkin-tree} \emph{(Motzkin trees)}
Motzkin trees are ordered trees where each internal vertex has outdegree 1 or 2. Thus the generating function $m$ for Motzkin trees satisfies the functional equation $m=1+zm+z^2m^2$ (see \seqnum{A001006}). Let $\tilde{m}=1+zm$ be the generating function for the planted Motzkin trees (i.e., Motzkin trees with the root of degree at most 1), then $\tilde{m}=1+z(1+(\tilde{m}-1)\tilde{m})=1+z(1-\tilde{m}+\tilde{m}^2)$, so in this case $\gamma=1-z+z^2$, a palindrome with $\dar(\gamma)=0+2=2$, and hence $(\tilde{m},z\tilde{m})$ is a pseudo-involution (see \seqnum{A086246}). We also note that $m$ is the generating function for the number of \emph{Motzkin paths}, i.e., paths from $(0,0)$ to $(n,0)$ with unit steps $(1,1)$, $(1,0)$, $(1,-1)$ staying on or above the $x$-axis, as well as for the number of (partial) \emph{noncrossing matchings} of points labeled $1,\dots,n$ (i.e., if $a<b<c<d$, then a matching cannot simultaneously include pairs $ac$ and $bd$), and $\tilde{m}$ is the generating function for the number of (partial) noncrossing matchings where $1$ is unmatched (or the matching is empty).
\end{example}

\begin{example} \label{ex:double-root} \emph{(Doubling at the root)}
Consider a class of colored $k$-ary trees where the root of any nonempty tree has 2 possible colors and all other vertices have only one color. The generating function for such $k$-ary trees is $t_k=2T_k-1$, where $T_k$ is as above. Then
\[
t_k=1+2(T_k-1)=1+2zT_k^k=1+2z\left(\frac{1+t_k}{2}\right)^k,
\]
so in this case $\gamma=2\left(\frac{1+z}{2}\right)^k$ is palindromic with $\dar(\gamma)=0+k=k$, and therefore $(t_k,zt_k^{k-1})$ is a pseudo-involution. In particular, this yields $\left(\frac{1+z}{1-z},z\right)$ for $k=1$, and $(2C-1,z(2C-1))$ for $k=2$ (\seqnum{A068875}).

Similarly, it is easy to check that the functional equation $g=1+z(1+g)^k$ yields $g=t_k(2^{k-1}z)$ and $f=zg^{k-1}$ for a pseudo-involution $(g,f)$.
\end{example}

\begin{example} \label{ex:double-leaf} \emph{(Doubling at the leaves)}
Consider a class of colored $k$-ary trees where the leaves of any nonempty tree have 2 possible colors and all internal vertices have only one color. The generating function $u_k$ for such trees satisfies the functional equation
\begin{equation} \label{eq:uk}
u_k=1+z(1+u_k^k),
\end{equation}
so in this case $\gamma=1+z^k$ is palindromic with $\dar(\gamma)=0+k=k$, and therefore $(u_k,zu_k^{k-1})$ is a pseudo-involution. Since $u_k=(1+z)+zu_k^k$, we obtain after some routine algebra that
\[
u_k=(1+z)T_k\big(z(1+z)^{k-1}\big).
\] 
In particular, this yields $\left(\frac{1+z}{1-z},z\right)$ for $k=1$, and $(u_2,zu_2)$ for $k=2$, where $u_2(z)=(1+z)C(z(1+z))$, or equivalently, $zu_2=(zC)\circ(z+z^2)$, with the coefficient sequence given by \seqnum{A025227}.
\end{example}

Note the following curious duality between the pseudo-involutory functions $zt_k^{k-1}$ and $zu_k^{k-1}$ from Examples \ref{ex:double-root} and \ref{ex:double-leaf}. Since $zT_k^{k-1}=\overline{(z(1-z)^{k-1})}$ and $1+zT_k^{k-1}=(2T_k-1)/T_k$, we have
\[
\begin{split}
zt_k^{k-1}&=(z(1+z)^{k-1})\circ\overline{(z(1-z)^{k-1})},\\
zu_k^{k-1}&=\overline{(z(1-z)^{k-1})}\circ(z(1+z)^{k-1}),
\end{split}
\]
and $z(1+z)^{k-1}=(-z)\circ(z(1-z)^{k-1})\circ(-z)$. Likewise, for the aerations $zt_k(z^{k-1})$ and $zu_k(z^{z-1})$, we have
\[
\begin{split}
zt_k(z^{k-1})&=(z+z^k)\circ\overline{(z-z^k)},\\
zu_k(z^{k-1})&=\overline{(z-z^k)}\circ(z+z^k).
\end{split}
\]
Furthermore, if $k$ is odd, then $z+z^k=(-z)\circ(z-z^k)\circ(-z)$, which implies that both $zt_k(z^{k-1})$ and $zu_k(z^{k-1})$ are pseudo-involutory.

\bigskip

Some well-known pseudo-involutions arise from functions $\gamma$ that are not polynomials.

\begin{example} \label{ex:catalan-cbc} \emph{(Central binomial coefficients)}
Consider $\gamma(z)=\frac{4z^2}{1+z}$, then $\gamma$ is palindromic as a ratio of two palindromic polynomials. Moreover, $\dar(\gamma)=(2+2)-(0+1)=3$, and if
\[
g=1+z\,\frac{4g^2}{1+g},
\]
then $g=\frac{1}{\sqrt{1-4z}}=B(z)=B$, the generating function for the central binomial coefficients. In this case, we obtain the pseudo-involution
\[
(B,zB^2)=\left(\frac{1}{\sqrt{1-4z}},\frac{z}{1-4z}\right).
\]
\end{example}

\begin{example} \label{ex:rna} \emph{(Secondary RNA structures)} A \emph{secondary RNA structure} can be modeled as a noncrossing partial matching of points on a straight line where no two adjacent points can be in the same pair. Let $g$ be the generating function of the counting sequence \seqnum{A004148} for the secondary RNA structures (see, e.g., Cameron and Nkwanta \cite{CN}), then $g$ satisfies the functional equation
\[
g=1+zg+z^2g(g-1).
\]
Then
\[
g-1=\frac{zg}{1-z^2g}=z\,\frac{g}{1-z^2g}\ ,
\]
so
\[
\frac{(1-z)^2}{z}\circ g=\frac{(g-1)^2}{g}=\frac{z^2g}{(1-z^2g)^2}=\frac{z}{(1-z)^2}\circ(z^2g),
\]
or, equivalently,
\[
z^2g=\overline{\left(\frac{z}{(1-z)^2}\right)}\circ\frac{(1-z)^2}{z}\circ g
=(zC^2(-z))\circ\frac{(1-z)^2}{z}\circ g.
\]
Therefore,
\[
1-z^2g=(1-zC^2(-z))\circ\frac{(1-z)^2}{z}\circ g = C(-z)\circ\frac{(1-z)^2}{z}\circ g=C\left(-\frac{(1-z)^2}{z}\right)\circ g,
\]
and hence,
\[
g=1+z\cdot\left(\frac{z}{C\left(-\frac{(1-z)^2}{z}\right)}\circ g\right)\!.
\]
To find the darga of
\[
\gamma=\frac{z}{C\left(-\frac{(1-z)^2}{z}\right)}\, ,
\]
note that $(1-z)^2/z$ is a generalized palindrome of darga $(0+2)-(1+1)=0$, and therefore
$C\!\left(-(1-z)^2/z\right)$ also has darga $0$. Thus, $\dar(\gamma)=\dar(z)-0=2$, and hence $(g,zg^{2-1})=(g,zg)$ is a pseudo-involution in the Bell subgroup.
\end{example}

When $\gamma$ is a quadratic polynomial, not necessarily a palindrome, we have the following general result.

\begin{theorem} \label{thm:quad}
For constants $a,b,c$, where $a\ne 0$, let $f\in\mathcal{F}_1$ satisfy the functional equation
\begin{equation} \label{eq:quad}
f=z(a+bf+cf^2),
\end{equation}
in other words, let $f$ have $A_f=a+bz+cz^2$. Let $t=1+(b/a)f=(a+bf)/a$. Then $h=zt$ is pseudo-involutory with $B_{h}=bC(acz)$.
\end{theorem}

Note that for $c=0$, $C(0)=1$ is the constant term of the power series expansion of $C(z)$, and indeed in that case $h=\frac{z}{1-bz}$ and $B_h=b$.

\smallskip

The way we find the $B$-function in the proof that follows is somewhat ad hoc. We will develop a more general method for finding $B$-functions in Section \ref{subsec:b-ogf-gen}, then apply it (see Example \ref{ex:thm-quad}) to recover this $B$-function result.

\begin{proof}
When $b=0$, we have $t=1$, $h=z$, and $B_h=B_z=0$. Assume $b\ne 0$. Substituting $f=(a/b)(t-1)$ into \eqref{eq:quad}, we obtain
\begin{equation} \label{eq:t-rec}
\begin{split}
t&=1+\frac{b}{a}f=1+\frac{b}{a}z(a+bf+cf^2)\\
&=1+\frac{b}{a}z\left(a+b\frac{a}{b}(t-1)+c\left(\frac{a}{b}(t-1)\right)^2\right)\\
&=1+z\left(bt+\frac{ac}{b}(t-1)^2\right),
\end{split}
\end{equation}
and notice that the $\gamma$ function for $t$ is
\begin{equation} \label{eq:gamma-t}
\gamma(z)=bz+\frac{ac}{b}(z-1)^2\in\mathcal{P}_2,
\end{equation}
since both summands are palindromic of darga 2. Therefore $h=zt^{2-1}=zt$ is pseudo-involutory.

Finally, for the $B$-function of $h=zt$, we have $zt=z+z^2tB_{zt}(z^2t)$, i.e., $B_{zt}(z^2t)=\frac{t-1}{zt}$. Let $g=g(z)=\frac{1}{b}B_{zt}(z^2t)$, then
\[
\begin{split}
g&=\frac{t-1}{bzt}=\frac{(b/a)f}{bzt}=\frac{f}{azt}=\frac{f/z}{at}=\frac{a+bf+cf^2}{a+bf}\\
&=1+\frac{cf^2}{a+bf}=1+\frac{cf^2}{at}
=1+acz^2t\left(\frac{f}{azt}\right)^2\\
&=1+\left(aczh\right)g^2,
\end{split}
\]
and $g=1+(aczh)g^2$ has a unique power series solution $g=C(aczh)$, which implies that
\[
B_{h}(z)=bC(acz). \qedhere
\]
\end{proof}

\begin{example} \label{ex:motzkin}
If $f=zm$, where $m$ is the generating function for the Motzkin numbers, then $a=b=c=1$, so $h=z(1+zm)=z\tilde{m}$ and $B_{h}=C$.
\end{example}

\begin{example} \label{ex:doubled-catalan}
If $f=C-1=zC^2$, then $f=z(1+2f+f^2)$, so $h=z(1+2f)=z(2C-1)=zt_2$ from Example \ref{ex:double-root} and $B_{h}=2C$.
\end{example}

\begin{example} \label{ex:k-motzkin}
If $f=zm_k$, where $m_k=1+kzm_k+z^2m_k^2$ is the generating function for the $k$-Motzkin numbers counting Motzkin paths with horizontal steps of $k$ colors. Then $f=z(1+kf+f^2)$, so $h=z(1+kf)=z(1+kzm_k)=z\tilde{m}_k$ and $B_{h}=kC$. In particular, letting $k=1$ and $k=2$ yields Examples \ref{ex:motzkin} and \ref{ex:doubled-catalan}, respectively, while letting $k=3$ yields the planted \emph{hex trees} \cite{DMR} (see \seqnum{A002212}).
\end{example}

\begin{example} \label{ex:doubled-binary-leaves}
If $f=u_2-1=z(1+u_2^2)$, where $u_2$ is as in Example \ref{ex:double-leaf}, then $f=z(2+2f+f^2)$, so $h=z(2+2f)/2=z(1+f)=zu_2$ and $B_{h}=2C(2z)$.
\end{example}

\begin{example} \label{ex:schroeder}
If $f=r-1$, where $r$ is the generating function for the large Schr\"{o}der numbers, then $r-1=z(r+r^2)$, i.e., $f=z(2+3f+f^2)$, so $h=z(2+3f)/2=z(3r-1)/2=z(r+s-1)$, where $s$ is the generating function for the little Schr\"{o}der numbers, and $B_{h}=3C(2z)$. The coefficient sequence of $t$ (that we may call \emph{huge Schr\"{o}der numbers}) is given by \seqnum{A238113}.
\end{example}

\begin{corollary} \label{cor:large-little}
Let $g=g(z)$ be the function defined by the functional relation
\[
g=1+z(a+bg+cg^2)
\]
for some constants $a,b,c$ such that $a+b+c\ne 0$. Let
\[
t=t(z)=\frac{(b+2c)g+(a-c)}{a+b+c}=1+\frac{b+2c}{a+b+c}(g-1).
\]
Then $zt$ is a pseudo-involutory, and its $B$-function is
\[
B_{zt}(z)=(b+2c)C\bigl((a+b+c)cz\bigr).
\]
\end{corollary}

\begin{proof} Let $f=r-1$ in Theorem \ref{thm:quad}.
\end{proof}

\begin{remark} \label{rem:motzkin}
In particular, as expected from the prior discussion, $t=g$ when $c=a$ in Corollary \ref{cor:large-little}, and then $zg$ is pseudo-involutory with $B_{zg}=(b+2c)C\bigl((b+2c)cz\bigr)$. For example, when $g=\tilde{m}$, we have $a=c=1$ and $b=-1$ in Corollary \ref{cor:large-little}, so $B_{z\tilde{m}}=C$, as before. 
\end{remark}

\begin{remark} \label{rem:schroeder}
Note also that when $a=0$, we obtain a generalization of large and little Schr\"{o}der number generating functions $r=r(z)$ and $s=s(z)$, respectively, where
\[
\begin{split}
r&=1+z(br+cr^2)=1+(b+c)zrs,\\
s&=\frac{b+cr}{b+c}=\frac{1}{1-czr}=1+czrs,\\
t&=r+s-1=1+(b+2c)zrs.
\end{split}
\]
The Schr\"{o}der numbers correspond to $b=c=1$, whereas the Catalan numbers correspond to $b=0$, $c=1$. In the latter case, $r=s=C$, so that $t=2r-1=2C-1$, and we recover the pseudo-involutory function $z(2C-1)$. To interpret this combinatorially, notice that, given some classes of trees enumerated by the coefficient sequences of $r$ and $s$, the corresponding structure for $t$ is obtained by taking the combinatorial sum of the $r$- and $s$-structures and identifying the two empty objects. The ``doubling'' mentioned above for $k$-ary trees generalizes in this case to the disjoint sum of the nonempty ``large'' and ``little'' objects. 

Another related combinatorial interpretation of $r$ and $s$ is $(b,c)$-Schr\"{o}der paths, which are Schr\"{o}der paths where level steps $L=(2,0)$ come in $b$ colors, while up-steps $U=(1,1)$ and down-steps $D=(1,-1)$ come in $c$ colors. The large $(b,c)$-Schr\"{o}der paths (corresponding to $r$) have no further restrictions, while the little $(b,c)$-Schr\"{o}der paths (corresponding to $s$) additionally have no level steps at height 0. 

Alternatively, $t$ counts the $(b,c)$-Schr\"{o}der paths, such that if the leftmost  $D$ comes before the leftmost $L$, then that leftmost $D$ has $2c$ colors, i.e., twice as many as other downsteps. To see this, note that any nonempty $(b,c)$-Schr\"{o}der path can be uniquely decomposed either as $U^{k-1}LP_1DP_2d\cdots DP_k$  or as $U^k\mathbf{D}P_1DP_2D\cdots DP_k$ for some positive integer $k$ and smaller $(b,c)$-Schr\"{o}der paths $P_1,\dots,P_k$. In the latter case, the leftmost $D$ (marked in bold) can be colored in $2c$, rather than $c$, colors. Then the number of such paths has the generating function
\[
1+\frac{bzr}{1-czr}+\frac{2czr}{1-czr}=1+(b+2c)zrs=t.
\]
\end{remark}

We conclude this section with two applications of Theorem \ref{thm:palindrome} to exponential Riordan group. Both of these examples are related to labeled trees.

\begin{example} \label{ex:labeled-trees} \emph{(Rooted labeled trees)} 
The set of labeled trees with vertices $\{0,1,\dots,n\}$, $n\ge 0$, rooted at $0$ are counted by the sequence $1,1,3,16,125,1296,\dots,(n+1)^{n-1},\dots$ (\seqnum{A000272} without the initial $1$), whose exponential generating function $T=T(z)$ satisfies the functional equation
\begin{equation} \label{eq:labeled-trees}
T=e^{zT},
\end{equation}
or, equivalently, $zT=ze^{zT}$, i.e., $zT=\overline{(ze^{-z})}$.
From Equation \eqref{eq:gamma-egf}, we have that $\gamma=z$, a palindrome with $\dar(\gamma)=1+1=2$. Therefore, from the second part of Theorem \ref{thm:palindrome}, it follows that $[T,zT^2]$ is a pseudo-involution in the 2-Bell subgroup of the exponential Riordan group. An application of Lagrange inversion shows that the companion $zT^2$ of $T$ is the exponential generating function of the sequence
\[
\left\{2n(n+1)^{n-2}\right\}_{n\ge 0}=\{0,1,4,24,200,2160,28812,\dots\}
\]
(a shifted version of \seqnum{A089946}). Similarly, the $k$th column of the matrix $[T,zT^2]$ has exponential generating function $T\cdot(zT^2)^k/k!=z^kT^{2k+1}/k!$, which, after applying Lagrange inversion and some algebraic manipulations, yields
\[
[T,zT^2]=\left[(2k+1)\binom{n}{k}(n+k+1)^{n-k-1}\right]_{n,k\ge 0},
\]
which begins with
\[
\begin{bmatrix*}[r]
1 & 0 & 0 & 0 & 0 & 0\\
1 & 1 & 0 & 0 & 0 & 0\\
3 & 6 & 1 & 0 & 0 & 0\\
16 & 45 & 15 & 1 & 0 & 0\\
125 & 432 & 210 & 28 & 1 & 0\\
1296 & 5145 & 3200 & 630 & 45 & 1
\end{bmatrix*}.
\]
Note that the associated Riordan array
\[
[1,zT^2]=\left[2k\binom{n}{k}(n+k)^{n-k-1}\right]_{n,k\ge 0}
\] 
in the Lagrange subgroup begins
\[
\begin{bmatrix*}[r]
1 & 0 & 0 & 0 & 0 & 0\\
0 & 1 & 0 & 0 & 0 & 0\\
0 & 4 & 1 & 0 & 0 & 0\\
0 & 24 & 12 & 1 & 0 & 0\\
0 & 200 & 144 & 24 & 1 & 0\\
0 & 2160 & 1960 & 480 & 40 & 1
\end{bmatrix*}.
\]
\end{example}

\begin{example} \label{ex:labeled-trees-subtrees}
\emph{(Rooted labeled trees with 2-colored leaves)} Consider a function $S=S(z)$ that satisfies a functional equation similar to \eqref{eq:labeled-trees}:
\begin{equation} \label{eq:labeled-trees-subtrees}
S=e^{z(1+S)}.
\end{equation}
Then $zS=ze^ze^{zS}$, or, in other words, $zS=(zT)\circ(ze^z)$, where $T$ is from Example \ref{ex:labeled-trees}. Here, the function $\gamma$ from Equation~\eqref{eq:gamma-egf} is $\gamma=1+z$, a palindrome with $\dar(\gamma)=0+1=1$, and thus $[S,zS]$ is a pseudo-involution in the Bell subgroup of the exponential Riordan group. It is easy to check directly that $zS$ is pseudo-involutory: note that
\begin{equation} \label{eq:labeled-trees-subtrees-pseudo}
zS=(zT)\circ(ze^z)=\overline{(ze^{-z})}\circ(ze^z)=(-z)\circ\overline{(ze^z)}\circ(-z)\circ(ze^z),
\end{equation}
and thus
\[
\overline{(zS)}=\overline{(ze^z)}\circ(-z)\circ(ze^z)\circ(-z)=(-z)\circ(zS)\circ(-z).
\]
The coefficient sequence of $S$ begins $1,2,8,56,576,7872,\dots$ (\seqnum{A349562}) and counts rooted labeled forests with 2-colored leaves. Equivalently, the sequence \seqnum{A349562} counts rooted labeled trees on vertices $0,1,\dots,n$ rooted at $0$ with 2-colored leaves (except the singleton tree $0$ is 1-colored). The coefficient sequence of $zS$ begins $0,1,4,24,224,2880,47232,\dots$ (see \seqnum{A216857} and \seqnum{A038049}) and counts, e.g., rooted labeled trees on vertices $1,2,\dots,n$ with 2-colored leaves (with the exception that the singleton tree is 1-colored). Moreover, Lagrange inversion for $zT$ and the identity $zS=(zT)\circ(ze^z)$ yield $\left[\frac{z^n}{n!}\right]zS=\sum_{j=1}^{n}{\binom{n}{j}j^{n-1}}$ and $\left[\frac{z^n}{n!}\right]S=\sum_{j=0}^{n}{\binom{n}{j}(j+1)^{n-1}}$, and similarly,
\[
\left[\frac{z^n}{n!}\right]\frac{S(zS)^k}{k!}
=\binom{n}{k}(k+1)\sum_{j=0}^{n-k}{\binom{n-k}{j}(j+k+1)^{n-k-1}}.
\]
Thus, the array
\[
[S,zS]=\left[\binom{n}{k}(k+1)\sum_{j=0}^{n-k}{\binom{n-k}{j}(j+k+1)^{n-k-1}}\right]_{n,k\ge 0}
\] 
begins
\[
\begin{bmatrix*}[r]
1 & 0 & 0 & 0 & 0 & 0\\
2 & 1 & 0 & 0 & 0 & 0\\
8 & 8 & 1 & 0 & 0 & 0\\
56 & 72 & 18 & 1 & 0 & 0\\
576 & 832 & 288 & 32 & 1 & 0\\
7872 & 12160 & 5040 & 800 & 50 & 1
\end{bmatrix*},
\]
and the associated array
\[
[1,zS]=\left[\binom{n}{k}k\sum_{j=0}^{n-k}{\binom{n-k}{j}(j+k)^{n-k-1}}\right]_{n,k\ge 0}
\]
in the Lagrange subgroup begins
\[
\begin{bmatrix*}[r]
1 & 0 & 0 & 0 & 0 & 0\\
0 & 1 & 0 & 0 & 0 & 0\\
0 & 4 & 1 & 0 & 0 & 0\\
0 & 24 & 12 & 1 & 0 & 0\\
0 & 224 & 144 & 24 & 1 & 0\\
0 & 2880 & 2080 & 480 & 40 & 1\\
\end{bmatrix*}.
\]
\end{example}

\section{$B$-sequence and $B$-function} \label{sec:b-seq}

As we mentioned earlier (see Equation \eqref{eq:b-seq-def} on page \pageref{eq:b-seq-def}), given a pseudo-involution $(g,f)$, its $B$-sequence $\{b_n\}_{n\ge 0}$ and $B$-function $B_f(z)=\sum_{n=0}^{\infty}{b_n z^n}$ are defined so that the pseudo-involutory function $f$ satisfies
\begin{equation} \label{eq:b-seq-def-recall}
f-z=(zB_f)\circ(zf)=zfB_f(zf)=\sum_{n=0}^{\infty}{b_n (zf)^{n+1}}.
\end{equation} 
Note that the $B_f$ depends only on $f$. In fact, the $B$-sequence of $f$, together with the leftmost column of $(g,f)$, gives us enough information to determine $f$ and the entire matrix $(g,f)$. For example, the $B$-function of $f=zC^3$ is $B_{zC^3}=3+z$, a fact that we shall re-establish later in this section, so its $B$-sequence is $3,1,0,0,0,\dots$. Looking at the matrix $(C,zC^3)$, whose first few rows and columns are given by
\[
\begin{bmatrix}
1 & 0 & 0 & 0 & 0 & 0 \\ 
1 & 1 & 0 & 0 & 0 & 0 \\ 
2 & 4 & 1 & 0 & 0 & 0 \\ 
5 & 14 & 7 & \mathbf{1} & 0 & 0\\
14 & \mathbf{48} & \mathbf{35} & 10 & 1 & 0 \\ 
42 & 165 & \mathbf{154} & 65 & 13 & 1
\end{bmatrix},
\]
we see that, for instance, $154=48+\mathbf{3}\cdot 35+\mathbf{1}\cdot 1$. Every
entry other than those in the leftmost column can be computed from this simple 3,1-rule.

\subsection{$B$-functions for ordinary Riordan arrays with palindromic $\gamma$} \label{subsec:b-ogf-gen}

Recall from Theorem~\ref{thm:palindrome}, that if $g=g(z)$ satisfies the functional equation~\eqref{eq:gamma-ogf}, where $\gamma=\gamma(z)$ is a generalized palindrome of darga $d$, then $(g,zg^{d-1})$ is a pseudo-involution. Then
\[
\dar\left(\frac{\gamma(z)}{z^{d/2}}\right)=d-2\cdot\frac{d}{2}=0,
\]
so $\frac{\gamma(z)}{z^{d/2}}$ is a function of $z+\frac{1}{z}$, or, equivalently, a function of $z+\frac{1}{z}-2=\frac{(z-1)^2}{z}$. Define a function $\delta=\delta(z)$ so that
\begin{equation} \label{eq:def-delta}
\delta\left(\frac{(z-1)^2}{z}\right)=\frac{\gamma(z)}{z^{d/2}} \ ,
\end{equation}
and suppose that $\delta(0)\ne 0$. Then let the function $D=D(z)$ be defined by
\begin{equation} \label{eq:def-d}
zD^2=\overline{\left(\frac{z}{\delta^2}\right)}.
\end{equation}
In other words, from Equation \eqref{eq:az-def}, we have $D^2 = 1/A_{z/\delta^2}$ and $\delta^2=A_{zD^2}$.

\begin{lemma} \label{lem:delta-d}
Let $g=g(z)$ satisfy the functional equation $g=1+z\gamma(g)$, where $\gamma=\gamma(z)$ is a generalized palindrome of darga $d$, and let the functions $\delta$ and $D$ be as in \eqref{eq:def-delta} and \eqref{eq:def-d}. Then
\begin{equation} \label{eq:delta-d}
\frac{(g-1)^2}{g}=(zD^2)\circ(z^2 g^{d-1}).
\end{equation}
\end{lemma}

\begin{proof}
We have
\[
\frac{(g-1)^2}{g}=\frac{(z\gamma(g))^2}{g}= z^2g^{d-1}\cdot \left(\frac{\gamma(g)}{g^{d/2}}\right)^2=z^2g^{d-1}\cdot\delta^2\left(\frac{(g-1)^2}{g}\right),
\]
or, equivalently,
\[
z^2g^{d-1}=\left(\frac{z}{\delta^2}\right)\circ\left(\frac{(g-1)^2}{g}\right),
\]
which implies \eqref{eq:delta-d}, as claimed.
\end{proof}

On the other hand, note that for $h=h(z)=zg^{d-1}$, we have $h-z=zhB_h(zh)$, and thus
\begin{equation} \label{eq:zbh_zh}
(zB_h^2)\circ(zh)=(zB_h^2)\circ(z^2g^{d-1})=\frac{(h-z)^2}{zh}=\frac{(g^{d-1}-1)^2}{g^{d-1}}=g^{d-1}+\frac{1}{g^{d-1}}-2,
\end{equation}
so the final piece we need to find $B_h$ is the family of polynomials $p_k=p_k(z)$, $k\ge 0$, such that
\begin{equation} \label{eq:p-sym}
(zp_k)\circ\left(\frac{(g-1)^2}{g}\right)=\frac{(g^{k+1}-1)^2}{g^{k+1}},
\end{equation}
or, equivalently,
\[
(zp_k)\circ\left(g+\frac{1}{g}-2\right)=g^{k+1}+\frac{1}{g^{k+1}}-2.
\]
In addition, consistent with the above definition, we let $p_{-1}(z)=0$. Recall that for Chebyshev polynomials of the first kind $T_k(z)$, $k\ge 0$, defined by $T_k(\cos\theta)=\cos(k\theta)$, we have the identity
\[
T_k\left(\frac{z+\frac{1}{z}}{2}\right)=\frac{z^k+\frac{1}{z^k}}{2}.
\]
Therefore,
\begin{equation} \label{eq:pk-from-T}
p_k(z)=\frac{1}{z}\left(2T_{k+1}\left(\frac{z+2}{2}\right)-2\right)=\left(\frac{T_{k+1}-1}{z-1}\right)\circ\left(\frac{z+2}{2}\right).
\end{equation}
Consider the Riordan array \seqnum{A156308}, 
\[
[d_{n,k}]_{n,k\ge 0}=\left(\frac{1+z}{(1-z)^3},\frac{z}{(1-z)^2}\right),
\]
where
\[
d_{n,k}=\frac{n+1}{k+1}\binom{n+k+1}{2k+1}=2\binom{n+k+2}{2k+2}-\binom{n+k+1}{2k+1}.
\]
Its first few rows are given by
\[
\begin{bmatrix}
1 & 0 & 0 & 0 & 0\\
4 & 1 & 0 & 0 & 0\\
9 & 6 & 1 & 0 & 0\\
16 & 20 & 8 & 1 & 0\\
25 & 50 & 35 & 10 & 1
\end{bmatrix}.
\]

\begin{lemma} \label{lem:alpha-nk}
The array $[d_{n,k}]_{n,k\ge 0}$ satisfies the identity
\begin{equation} \label{eq:alpha-nk}
\left[d_{n,k}\right]_{n,k\ge 0}\left[\left(\frac{(g-1)^2}{g}\right)^{k+1}\right]_{k\ge 0}^T=\left[\frac{(g^{n+1}-1)^2}{g^{n+1}}\right]_{n\ge 0}^T
\end{equation}
(where the superscript ${}^T$ denotes the transpose).
\end{lemma}

\begin{proof}
The proof of this identity follows from the fundamental theorem of Riordan arrays \eqref{eq:ftra}, since
\[
\left(\frac{1+z}{(1-z)^3},\frac{z}{(1-z)^2}\right)*\frac{\frac{(g-1)^2}{g}}{1-\frac{(g-1)^2}{g}z}=\frac{g}{1-zg}+\frac{1/g}{1-z/g}-\frac{2}{1-z}. \qedhere
\]
\end{proof}

Thus, the polynomials $p_k(z)$ of \eqref{eq:p-sym} are given by
\begin{equation} \label{eq:poly-p}
p_k(z)=\sum_{j=0}^{k}{d_{k,j}z^j}.
\end{equation}

Moreover, letting $z=\cos\theta$ and applying routine trigonometric identities shows that
\[
\begin{split}
\frac{T_{2l+1}(z)-1}{z-1}&=(U_l(z)+U_{l-1}(z))^2,\\
\frac{T_{2l+2}(z)-1}{z-1}&=2(1+z)U_l(z)^2,
\end{split}
\]
where $U_l(z)$, $l\ge 0$, defined by $U_l(\cos\theta)=\frac{\sin((l+1)\theta)}{\sin\theta}$, are the Chebyshev polynomials of the second kind.
Therefore, polynomials $p_k$ can be also expressed in terms of the polynomials $U_l$ as follows (see also the notes for \seqnum{A156308}):
\begin{equation} \label{eq:pk_cheb}
\begin{split}
p_{2l}(z) &= \left(U_l\left(\frac{z + 2}{2}\right) + U_{l-1}\left(\frac{z + 2}{2}\right)\right)^2,\\
p_{2l+1}(z) &= (z+4)\left(U_l\left(\frac{z + 2}{2}\right)\right)^2,
\end{split}
\end{equation}
where we note that $U_{-1}(z)=0$. A Riordan array $[c_{k,n}]_{k,n\ge 0}$ where the generating function for row $k\ge 0$ is $U_k\!\left(\frac{z}{2}\right)$ is given by \seqnum{A049310}, in other words,
\[
\left[c_{k,n}\right]_{k,n\ge 0}=\left(\frac{1}{1+y^2}\, ,\frac{y}{1+y^2}\right),
\]
and
\begin{equation} \label{eq:cheb2}
c_{k,n}=
\begin{cases}
\displaystyle{(-1)^{\frac{k+n}{2}}\binom{\frac{k+n}{2}}{n}}, & \text{if} \ k-n\ge 0 \ \text{is even};\\
0, & \text{otherwise}.
\end{cases}
\end{equation}

The preceding discussion, including Lemmas \ref{lem:delta-d} and \ref{lem:alpha-nk}, implies the following theorem.

\begin{theorem} \label{thm:bh-d}
Let $g=g(z)$ satisfy the functional equation $g=1+z\gamma(g)$, where $\gamma=\gamma(z)$ is a generalized palindrome of darga $d$, and let the functions $\delta$ and $D$ be as in \eqref{eq:def-delta} and \eqref{eq:def-d}. Then
\begin{equation} \label{eq:bh-d}
B_h = D \sqrt{p_{|d-1|-1}(zD^2)} =
\begin{cases}
D \sqrt{p_{d-2}(zD^2)}, & \text{if} \quad d\ge 2;\\
0, & \text{if} \quad d=1;\\
D \sqrt{p_{-d}(zD^2)}, & \text{if} \quad d\le 0.
\end{cases}
\end{equation}
\end{theorem}

\begin{proof}
Indeed, if $d\ge 2$, then from Lemma \ref{lem:delta-d} and Equations \eqref{eq:zbh_zh} and \eqref{eq:p-sym} we have
\[
zB_h^2=(zp_{(d-1)-1})\circ(zD^2)=(zp_{d-2})\circ(zD^2).
\] 
If $d=1$, then it is easy to see that $h=f=z$ and thus $B_h=0$. If $d\le 0$, then we note that
\[
g^{d-1}+\frac{1}{g^{d-1}}-2=g^{1-d}+\frac{1}{g^{1-d}}-2,
\]
so, for $d\le 0$, we have
\[
B_h = D \sqrt{p_{(1-d)-1}(zD^2)}=D \sqrt{p_{-d}(zD^2)}. \qedhere
\]
\end{proof}

\begin{remark} \label{rem:zbh2-explicit}
We can combine the steps involved in finding, successively, $\delta$, $D$, $p_k$, and $zB_h^2$ into explicit formulas for $zD^2$ and $zB_h^2$ that involve only $\gamma(z)$, namely,
\begin{equation}
\begin{split}
zD^2&=\frac{(z-1)^2}{z}\circ\overline{\left(\frac{z^{d-1}(z-1)^2}{\gamma^2(z)}\right)},\\
zB_h^2&=\frac{(z^{d-1}-1)^2}{z^{d-1}}\circ\overline{\left(\frac{z^{d-1}(z-1)^2}{\gamma^2(z)}\right)}.
\end{split}
\end{equation}
However, we believe that for practical calculations, breaking down the procedure into steps as described above is more effective.
\end{remark}

\begin{example} \label{ex:thm-quad}
Let us apply Theorem \ref{thm:bh-d} to re-establish Theorem \ref{thm:quad}. The function $t$ in Theorem \ref{thm:quad} satisfies the equation $t=1+z\gamma(t)$, where $\gamma(z)=bz+\frac{ac}{b}(z-1)^2$ (see \eqref{eq:gamma-t}) is palindromic of darga $d=2$. Thus, $\delta\left(\frac{(z-1)^2}{z}\right)=\frac{\gamma(z)}{z}=b+\frac{ac}{b}\frac{(z-1)^2}{z}$, and hence, $\delta=\delta(z)=b+\frac{ac}{b}z$. Therefore,
\[
\frac{z}{\delta^2}=\frac{z}{(b+\frac{ac}{b}z)^2}=\frac{\frac{z}{b^2}}{(1+ac\frac{z}{b^2})^2}=\frac{z}{ac}\circ\frac{z}{(1+z)^2}\circ\frac{acz}{b^2},
\]
so
\[
zD^2=\overline{\left(\frac{z}{\delta^2}\right)}=\frac{b^2z}{ac}\circ (zC^2) \circ (acz)=b^2zC^2(acz),
\]
and thus $D=bC(acz)$, so $B_{zt}=D\sqrt{p_0(zD^2)}=D=bC(acz)$.
\end{example}

\subsection{$B$-functions for exponential Riordan arrays with palindromic $\gamma$} \label{subsec:b-egf-gen}

Recall from Theorem~\ref{thm:palindrome}, that if $g=g(z)$ satisfies the functional equation~\eqref{eq:gamma-egf}, where $\gamma=\gamma(z)$ is a generalized palindrome of darga $d$, then the exponential Riordan array $[g,zg^d]$ is a pseudo-involution (as is the ordinary Riordan array $(g,zg^d)$). As in Section \ref{subsec:b-ogf-gen}, we have 
\[
\dar\left(\frac{\gamma(z)}{z^{d/2}}\right)=d-2\cdot\frac{d}{2}=0.
\]
Define a function $\varepsilon=\varepsilon(z)$ so that
\begin{equation} \label{eq:def-epsilon}
\varepsilon\left(\log(z)\right)=\frac{\gamma(z)}{z^{d/2}}\ ,\qquad  \text{i.e.} \qquad \varepsilon(z)=\frac{\gamma(e^z)}{e^{dz/2}}\ ,
\end{equation}
and suppose that $\varepsilon(0)=\gamma(1)\ne 0$. Then let the function $E=E(z)$ be defined by
\begin{equation} \label{eq:def-e}
zE=\overline{(z/\varepsilon)}=\overline{\left(\frac{ze^{dz/2}}{\gamma(e^z)}\right)}.
\end{equation}
In other words, from Equation \eqref{eq:az-def}, we have $E = 1/A_{z/\varepsilon}$ and $\varepsilon=A_{zE}$. Then we have the following lemma.

\begin{lemma} \label{lem:epsilon-e}
Let $g=g(z)$ satisfy the functional equation $g=e^{z\gamma(g)}$, where $\gamma=\gamma(z)$ is a generalized palindrome of darga $d$, and let the functions $\delta$ and $D$ be as in \eqref{eq:def-epsilon} and \eqref{eq:def-e}. Then
\begin{equation} \label{eq:epsilon-e}
\log g=(zE)\circ\sqrt{z}\circ(z^2 g^d).
\end{equation}
\end{lemma}

\begin{proof}
This follows immediately from the fact that
\[
\frac{z}{\varepsilon}\circ \log g = \frac{\log g}{\varepsilon(\log g))}=\frac{z\gamma(g)}{\gamma(g)/g^{\frac{d}{2}}}=zg^{d/2}=\sqrt{z^2g^d}. \qedhere
\]
\end{proof}

This lets us find the $B$-function of the pseudo-involutory companion $h=zg^d$ of $g$.

\begin{theorem} \label{thm:bh-e}
Let $g=g(z)$ satisfy the functional equation $g=e^{z\gamma(g)}$, where $\gamma=\gamma(z)$ is a generalized palindrome of darga $d$, let $h=zg^d$, and let the functions $\varepsilon$ and $E$ be as in \eqref{eq:def-epsilon} and \eqref{eq:def-e}. Then
\begin{equation} \label{eq:bh-e}
B_h = \left(\frac{2\sinh\left(\frac{d}{2}zE\right)}{z}\right)\circ\sqrt{z}.
\end{equation}
\end{theorem}

In other words, $B_h(z)=\sum_{j=0}^{\infty}{\beta_j\frac{z^j}{(2j+1)!}}$, where $\beta_j$ is the coefficient at $\frac{z^{2n+1}}{(2n+1)!}$ in the power series expansion of $2\sinh(dzE/2)$.

\begin{proof}
As before, we have
\[
(zB_h^2)\circ(z^2g^d)=\frac{(h-z)^2}{zh}=\frac{(g^d-1)^2}{g^d}=\frac{(e^{dz}-1)^2}{e^{dz}}\circ (\log g)=4\sinh^2(dz/2)\circ(\log g).
\]
Therefore, from Lemma \ref{lem:epsilon-e}, we obtain
\[
zB_h^2 = 4\sinh^2(dz/2)\circ(zE)\circ\sqrt{z},
\]
or, equivalently,
\[
B_h=\frac{(2\sinh(dz/2))\circ(zE)\circ\sqrt{z}}{\sqrt{z}}=\left(\frac{2\sinh\left(\frac{d}{2}zE\right)}{z}\right)\circ\sqrt{z}. \qedhere
\]
\end{proof}

\begin{example} \label{ex:labeled-trees-b}
Consider the set of labeled trees of Example \ref{ex:labeled-trees} with vertices $\{0,1,\dots,n\}$, $n\ge 0$ and rooted at $0$. The number of such trees on $n+1$ vertices is $(n+1)^{n-1}$, and their exponential generating function $T=T(z)$ is given by the equation $T=e^{zT}$. Thus, in this case, we have $\gamma(z)=z$ of darga $d=1+1=2$, so the pseudo-involutory companion of $T$ is $h=zT^2$, and
\[
\varepsilon(z)=\frac{e^z}{e^z}=1, \qquad zE=\overline{(z/\varepsilon)}=z,
\]
so
\[
B_h=\frac{2\sinh(z)}{z}\circ\sqrt{z}=\frac{2\sinh(\sqrt{z})}{\sqrt{z}}=2\sum_{j=0}^{\infty}{\frac{z^j}{(2j+1)!}}\ ,
\]
and thus $\beta_j=2$ for all $j\ge 0$. We saw in Example \ref{ex:labeled-trees} that the exponential Riordan arrays $[T,zT^2]$ and $[1,zT^2]$ are given by
\[
\begin{split}
[T,zT^2]&=\left[(2k+1)\binom{n}{k}(n+k+1)^{n-k-1}\right]_{n,k\ge 0}\ ,\\
[1,zT^2]&=\left[2k\binom{n}{k}(n+k)^{n-k-1}\right]_{n,k\ge 0}\ ,
\end{split}
\]
so their reductions (as in Definition \ref{def:exp-red}) are
\[
\begin{split}
\red[T,zT^2]&=\left[(2k+1)(n+k+1)^{n-k-1}\right]_{n\ge k\ge 0}\ ,\\
\red[1,zT^2]&=\left[2k(n+k)^{n-k-1}\right]_{n\ge k\ge 0}\ ,
\end{split}
\]
with all entries above the main diagonal equal to $0$. (Here we evaluate the entry at $(n,k)=(0,0)$ of $[1,zT^2]$ and $\red[1,zT^2]$ as $1$, similarly to other entries on the main diagonal where $n=k>0$.) Substituting these values into the recurrence relation \eqref{eq:b-seq-exp-def}, we obtain the following respective identities for $n\ge k\ge 0$:
\[
\begin{split}
(2k+3)(n+k+3)^{n-k-1}&\\
=(2k+1)(n+k+1)&^{n-k-1}+2\sum_{j=0}^{\infty}{\binom{n-k}{2j+1}(2k+2j+3)(n+k+2)^{n-k-2j-2}},\\
(k+1)(n+k+2)^{n-k-1}&=k(n+k)^{n-k-1}+2\sum_{j=0}^{\infty}{\binom{n-k}{2j+1}(k+j+1)(n+k+1)^{n-k-2j-2}},
\end{split}
\]
where we canceled the factor of $2$ in the second identity. It would be interesting to see a combinatorial interpretation of the above identities.
\end{example}

\begin{example} \label{ex:labeled-trees-subtrees-b}
Consider the function $S$ of Example \ref{ex:labeled-trees-subtrees} defined by $S=e^{z(1+S)}$. Then $\gamma(z)=1+z$ is palindromic with darga $d=0+1=1$, so the pseudo-involutory companion of $S$ is $h=zS$. Moreover,
\[
\varepsilon=\frac{ze^{z/2}}{1+e^z}=\frac{2z}{2\cosh(z/2)}=(z/2)\sech(z/2)=(z\sech(z))\circ(z/2),
\]
so
\[
zE=\overline{(z\sech(z))\circ(z/2)}=2\,\overline{(z\sech(z))}
\]
and
\[
B_h=\frac{2\sinh(zE/2)}{z}\circ\sqrt{z}=\frac{2\sinh(z)\circ\overline{(z\sech(z))}}{z}\circ\sqrt{z}
\]
As in the previous example, we want
\[
\beta_j=\left[\frac{z^j}{(2j+1)!}\right]B_h=\left[\frac{z^{2j+1}}{(2j+1)!}\right]\left(2\sinh(z)\circ\overline{(z\sech(z))}\right).
\]
This can be found using Lagrange inversion, as we need the odd-power coefficients of $2\sinh(u(z))$ where $u=u(z)$ satisfies the equation $u=z\cosh(u)$. After some routine manipulations, we obtain
\[
\beta_j=
\begin{cases}
2, & \text{if } j=0;\\
\sum_{i=0}^{j}{\binom{2j+2}{i}(j+1-i)^{2j}}, & \text{if } j\ge 1.
\end{cases}
\]
In other words, $\beta_j=2\cdot\seqnum{A007106}(j+1)$, that is twice the number of labeled trees on $2j+2$ vertices with all degrees odd. 

We also saw in Example \ref{ex:labeled-trees-subtrees} that the exponential Riordan arrays $[S,zS]$ and $[1,zS]$ are given by
\[
\begin{split}
[S,zS]&=\left[\binom{n}{k}(k+1)\sum_{j=0}^{n-k}{\binom{n-k}{j}(j+k+1)^{n-k-1}}\right]_{n,k\ge 0}\ ,\\
[1,zS]&=\left[\binom{n}{k}k\sum_{j=0}^{n-k}{\binom{n-k}{j}(j+k)^{n-k-1}}\right]_{n,k\ge 0}\ ,
\end{split}
\]
so their reductions (as in Definition \ref{def:exp-red}) are
\[
\begin{split}
\red[S,zS]&=
\left[(k+1)\sum_{j=0}^{n-k}{\binom{n-k}{j}(j+k+1)^{n-k-1}}\right]_{n\ge k\ge 0}\ ,\\
\red[1,zS]&=
\left[k\sum_{j=0}^{n-k}{\binom{n-k}{j}(j+k)^{n-k-1}}\right]_{n\ge k\ge 0}\ ,
\end{split}
\]
with all entries above the main diagonal equal to $0$. The matrix $\red[S,zS]$ begins
\[
\begin{bmatrix*}[r]
1 & 0 & 0 & 0 & 0 & 0\\
2 & 1 & 0 & 0 & 0 & 0\\
8 & 4 & 1 & 0 & 0 & 0\\
56 & 24 & 6 & 1 & 0 & 0\\
576 & 208 & 48 & 8 & 1 & 0\\
7872 & 2432 & 504 & 80 & 10 & 1
\end{bmatrix*},
\]
and the matrix $\red[1,zS]$ is simply $\red[S,zS]$ shifted 1 row down and 1 column to the right, with the new first column $[1,0,0,0,\dots]^T$ prepended.

As in the Example \ref{ex:labeled-trees-b}, it would be interesting to find a combinatorial interpretation of the recurrence \eqref{eq:b-seq-exp-def} for the exponential Riordan arrays $[S,zS]$ and $[1,zS]$.
\end{example}

\subsection{$B$-sequence and aeration} \label{subsec:b-aer}

Even though Theorem \ref{thm:bh-d} and Theorem \ref{thm:bh-e} give general solutions for $B$-functions, that general solution can be simplified for some special families, such as when $g$ satisfies \eqref{eq:gamma-ogf} and $d$ is even, or when $g$ satisfies \eqref{eq:gamma-egf} and $d$ is odd. Then we see from Corollary~\ref{cor:palindrome} that $(g(z^{d-1}),zg(z^{d-1}))$, or, respectively, $(g(z^d),zg(z^d))$, is also a pseudo-involution. The same applies to the corresponding exponential Riordan arrays. We shall see now how the $B$-functions of these un-aerated and aerated Riordan arrays (both ordinary and exponential) are related.

\begin{theorem}[``The Twin Powers Theorem''] \label{thm:aeration}
Let $g=g(z)$ be a power series with $g(0)\ne0$, and suppose that, for some integer $l\ge 0$,
\begin{equation} \label{eq:fh-def}
f=zg(z^{2l+1}) \quad \text{or} \quad h=zg^{2l+1}(z)
\end{equation}
is pseudo-involutory. Then both functions $f$ and $h$ are pseudo-involutory, and
\begin{equation} \label{eq:bf-bh}
(zB_h^2)\circ(z^{2l+1})=(zP_l^2)\circ(zB_f^2),
\end{equation}
where
\begin{equation} \label{eq:pl-def}
P_l=P_l(z)=\sum_{j=0}^{l}{a_{lj}z^j}
\end{equation}
and
\begin{equation} \label{eq:alj-def}
a_{lj}=\frac{2l+1}{2j+1}\binom{l+j}{2j}, \quad j,l\ge 0.
\end{equation}
\end{theorem}

Note that we only need to know that one of $f$ or $h$ is a pseudo-involution, then so is the other one. Note also that $a_{lj}=0$ if $j>l$, so $[a_{lj}]_{l,j\ge 0}$ is a lower triangular matrix. In fact, it is straightforward to see that
\[
[a_{lj}]_{l,j\ge 0}=\left(\frac{1+z}{(1-z)^{2}},\frac{z}{(1-z)^{2}}\right)
\]
is the array \seqnum{A111125} that starts with
\[
\begin{bmatrix}
1 & 0 & 0 & 0 & 0\\ 
3 & 1 & 0 & 0 & 0\\ 
5 & 5 & 1 & 0 & 0\\ 
7 & 14 & 7 & 1 & 0\\ 
9 & 30 & 27 & 9 & 1\\
\end{bmatrix}.
\] 
Then $P_l$ is the polynomial whose coefficient sequence forms the $l$-th row of $[a_{lj}]_{l,j\ge 0}$. For our proof, we will need a generalized version of this array, $[a_{ij}(uv)^{i-j}]_{i,j\ge 0}$.

\begin{lemma} \label{lem:alj}
The array $[a_{ij}(uv)^{i-j}]_{i,j\ge 0}$ satisfies the identity
\begin{equation} \label{eq:aij}
\left[a_{lj}(uv)^{l-j}\right]_{l,j\ge 0}\left[(u-v)^{2j+1}\right]_{j\ge 0}^T=\left[u^{2l+1}-v^{2l+1}\right]_{l\ge 0}^T
\end{equation}
(where the superscript ${}^T$ denotes the transpose).
\end{lemma}

This identity has a long history going back to Albert Girard (1629) and Edward Waring (1762). The article of Gould \cite{Gould}, where it appears as \cite[Equation (1)]{Gould} (up to a simple change of variables), discusses this history and generalizations to more variables.

The first few rows of \eqref{eq:aij} are
\[
\begin{bmatrix}
1 & 0 & 0 & 0 & 0 \\ 
3uv & 1 & 0 & 0 & 0 \\ 
5u^2v^{2} & 5uv & 1 & 0 & 0 \\ 
7u^3v^{3} & 14u^2v^{2} & 7uv & 1 & 0 \\ 
9u^4v^{4} & 30u^3v^{3} & 27u^2v^{2} & 9uv & 1
\end{bmatrix}
\begin{bmatrix}
\phantom{(}u-v\phantom{)^9} \\ 
\left( u-v\right) ^{3} \\ 
\left( u-v\right) ^{5} \\ 
\left( u-v\right) ^{7} \\ 
\left( u-v\right) ^{9}
\end{bmatrix}
=
\begin{bmatrix}
u-v \\ 
u^3-v^{3} \\ 
u^5-v^{5} \\ 
u^7-v^{7} \\ 
u^9-v^{9}
\end{bmatrix}.
\]
\begin{proof}
The proof of this matrix identity follows from the fundamental theorem of Riordan arrays \eqref{eq:ftra}, since
\[
\left( \frac{1+uvy}{(1-uvy)^2},\frac{y}{(1-uvy)^2}\right)*\frac{u-v}{1-(u-v)^2 y}=\frac{u}{1-u^2y}-\frac{v}{1-v^{2}y}
\]
(we use $y$ instead of $z$ as the main variable here due to substitutions that we will make later).
\end{proof}

\begin{proof}[Proof of Theorem~\ref{thm:aeration}]
Note that
\[
z^{2l+1}\circ f = f^{2l+1}=z^{2l+1}g^{2l+1}(z^{2l+1})=h(z^{2l+1}) = h \circ z^{2l+1},
\]
and therefore $h-z=zhB_h(zh)=(zB_h)\circ(zh)$ implies that
\[
\begin{split}
f^{2l+1}-z^{2l+1}&=h(z^{2l+1})-z^{2l+1}\\
&=(zB_h)\circ\left(z^{2l+1}h(z^{2l+1})\right)=(zB_h)\circ(z^{2l+1}f^{2l+1})=(zB_h)\circ\left((zf)^{2l+1}\right).
\end{split}
\]
Note also that $f-z=(zB_f)\circ(zf)=zfB_f(zf)$. Then, from equation \eqref{eq:aij} with $u=f$ and $v=z$, we have
\[
\begin{split}
(zf)^{2l+1}B_h\left((zf)^{2l+1})\right)
&=f^{2l+1}-z^{2l+1}\\
&=\sum_{j=0}^{l}{a_{lj}(zf)^{l-j}(f-z)^{2j+1}}\\
&=\sum_{j=0}^{l}{a_{lj}(zf)^{l-j}(zf)^{2j+1}(B_f(zf))^{2j+1}}\\
&=(zf)^{l+1}\sum_{j=0}^{l}{a_{lj}(zf)^{j}(B_f(zf))^{2j+1}}.
\end{split}
\]
Cancelling $(zf)^{l+1}$ on both sides and replacing $zf$ with $z$, we obtain
\[
z^l B_h(z^{2l+1})=\sum_{j=0}^{l}{a_{lj}z^jB_f(z)^{2j+1}}.
\]
Now, multiplying both sides by $\sqrt{z}$, we obtain the following compositions of functions on the respective sides of the equation:
\begin{equation} \label{eq:almost-there}
\left(\sqrt{z}B_h\right)\circ (z^{2l+1}) = \left(\sqrt{z}\sum_{j=0}^{l}{a_{lj}z^j}\right)\circ (zB_f^2)=(\sqrt{z}P_l)\circ(zB_f^2).
\end{equation}
Finally, squaring both sides of \eqref{eq:almost-there}, we obtain
\[
(zB_h^2)\circ(z^{2l+1})=(zP_l^2)\circ(zB_f^2). \qedhere
\]

\begin{example} \emph{($k$-ary trees)} \label{ex:b-cat}
Let $g=T_k$ from Example \ref{ex:catalan-tree}, then $d=2k$, so $2l+1=d-1=2k-1$, i.e., $l=k-1$. Then $f=zT_k(z^{2k-1})$, $h=zT_k^{2k-1}$, and
\begin{multline*}
f-z=zT_k(z^{2k-1})-z=z(T_k(z^{2k-1})-1)=\\=z\cdot z^{2k-1}T_k^k(z^{2k-1})=\left(z^2 T_k(z^{2k-1})\right)^k=(zf)^k=(zB_f)\circ(zf),
\end{multline*}
so $zB_f=z^k$, and thus, $B_f=z^{k-1}$ and $zB_f^2=z^{2k-1}$. Therefore, from Theorem~\ref{thm:aeration},
\[
(zB_h^2)\circ(z^{2k-1})=(zP_{k-1}^2)\circ(z^{2k-1}),
\]
and thus, $B_h=P_{k-1}$. For example, $k=2$ yields $d=4$, $l=1$, $g=C$, $f=zC(z^3)$, $h=zC^3$, $B_f=z$, and $B_h=3+z$. Likewise, for $k=3$, we have $B_{zT_3(z^5)}=z^2$ and $B_{zT_3^5}=5+5z+z^2$.
\end{example}

\end{proof}

We conclude this section with a few remarks relating the results of Theorems \ref{thm:bh-d} and \ref{thm:aeration}.

\begin{remark} \label{rem:p-P}
Notice that the Riordan arrays $\left[\alpha_{nk}\right]_{n,k\ge 0}$ and $\left[a_{nk}\right]_{n,k\ge 0}$ only differ by a factor of $\frac{1}{1-z}$ in the first component. This implies that
\begin{equation} \label{eq:p-sum}
p_k(z)=\sum_{l=0}^{k}{P_l(z)}, \quad \text{i.e.} \quad P_k(z)=p_k(z)-p_{k-1}(z).
\end{equation}
Moreover, it is easy to verify, using the fact that $T_k(\cos\theta)=\cos(k\theta)$ and Equation \eqref{eq:pk-from-T}, that
\begin{equation} \label{eq:p-root}
P_l(z)=U_l\left(\frac{z + 2}{2}\right) + U_{l-1}\left(\frac{z + 2}{2}\right)
=\sqrt{p_{2l}(z)}, \quad l\ge 0.
\end{equation}
\end{remark}

\begin{corollary} \label{cor:aeration_bf_d}
Let $g$, $f$, $h$, $l$ be as in Theorem \ref{thm:aeration}, and let $\delta$ and $D$ be as in \eqref{eq:def-delta} and \eqref{eq:def-d}. Then
\[
B_h= DP_{l}(zD^2) \qquad \text{and} \qquad B_f=z^l D(z^{2l+1}).
\]
\end{corollary}

\begin{proof}
Equation \eqref{eq:p-root} implies that for even darga $d=2l+2$, $l\ge 0$, in Theorem \ref{thm:bh-d}, we have
\[
B_h=D\sqrt{p_{2l}(zD^2)}=DP_{l}(zD^2),
\]
or, equivalently,
\[
zB_h^2=(zP_{l}^2)\circ(zD^2).
\]
Therefore, from Equation \eqref{eq:bf-bh} we obtain
\[
zB_f^2 = (zD^2)\circ (z^{2l+1}),
\]
and thus,
\[
B_f=z^l D(z^{2l+1}). \qedhere
\]
\end{proof}

\subsection{$B$-functions for doubled $k$-ary trees} \label{subsec:b-double-root}

Consider the family $t_k=t_k(z)$, $k\ge 1$, from Example \ref{ex:double-root}. Recall that $t_k=2T_k-1$, where $T_k=1+zT_k^k$ is the generating function for $k$-ary trees, and
\[
t_k=1+2z\left(\frac{1+t_k}{2}\right)^k.
\]
In this case, $\gamma=2\left(\frac{1+z}{2}\right)^k$ from Equation \eqref{eq:gamma-ogf} has darga $d=k$, so $(t_k,zt_k^{k-1})$ is a pseudo-involution in the $(k-1)$-Bell subgroup. Moreover, if $k$ is even (i.e., $k-1$ is odd), then by Corollary \ref{cor:palindrome}, $(t_k(z^{k-1}),zt_k(z^{k-1}))$ is a pseudo-involution in the Bell subgroup. 

Consider the $\delta$ and $D$ functions for $t_k$. From Equations \eqref{eq:def-delta} and \eqref{eq:def-d}, we have
\[
\delta^2\left(\frac{(z-1)^2}{z}\right)=\frac{\gamma^2(z)}{z^d}=\frac{4(1+z)^{2k}}{2^{2k}z^k}=4\left(\frac{(1+z)^2}{4z}\right)^k=4\left(1+\frac{(z-1)^2}{4z}\right)^k,
\]
so
\[
\delta^2(z)=4\left(1+\frac{z}{4}\right)^k,
\]
and hence,
\[
\frac{z}{\delta^2(z)}=\frac{z}{4\left(1+\frac{z}{4}\right)^k}=\frac{z}{(1+z)^k}\,\circ\, \frac{z}{4}.
\]
Therefore,
\[
zD^2=\overline{\left(\frac{z}{\delta^2}\right)}=(4z)\circ\overline{\left(\frac{z}{(1+z)^k}\right)}=(4z)\circ(zT_k^k)=4zT_k^k,
\]
and thus
\begin{equation} \label{eq:d_double_tk}
D=2T_k^{k/2}.
\end{equation}

From the discussion above, we see that for even $k=2l$, the $B$-function for $f=zt_{2l}(z^{2l-1})$ is quite simple.

\begin{theorem} \label{thm:b-aerated-double-root}
For $f=zt_{2l}(z^{2l-1})$, we have $B_f=2z^{l-1}T_{2l}^l(z^{2l-1})$.
\end{theorem}

\begin{proof}
We have $B_f=z^{l-1}D(z^{2l-1})$ from Corollary \ref{cor:aeration_bf_d}, and $D=2T_{2l}^{2l/2}=2T_{2l}^l$ from \eqref{eq:d_double_tk}.
\end{proof}

For example, when $k=2$, i.e., $l=1$, we get $2l-1=1$, $f=zt_2$, and $B_f=2T_2=2C$, as in Example~\ref{ex:doubled-catalan}. Similarly, for $k=4$, we have $l=2$, $2l-1=3$, $f=zt_4(z^3)$, and $B_f=2zT_4^2(z^3)$.

Note that it is possible to use Corollary \ref{cor:aeration_bf_d} to find $B_h$ for $h=zt_{2l}^{2l-1}$. For the remaining half of the family, where $h=zt_{2l+1}^{2l}$, $l\ge 0$, we find the $B$-functions directly from Theorem \ref{thm:bh-d}. 

\begin{lemma} \label{lem:b-double-root}
For $h=zt_{2l}^{2l-1}$, we have
\[
B_h = 2T_{2l}^l P_{l-1}\left(4zT_{2l}^{2l}\right)=2T_{2l}^l \cdot (U_{l-1}+U_{l-2})\circ(t_{2l}).
\]
For $h=zt_{2l+1}^{2l}$, we have
\[
B_h = 2T_{2l+1}^{l+1}\ U_{l-1}(t_{2l+1}).
\]
\end{lemma}

\begin{proof}
Note that $zD^2=4zT_k^k$ for $D=T_k^{k/2}$. We will make use of the following facts:
\[
\begin{split}
\frac{z+2}{2}\circ(4zT_k^k)&=\frac{4zT_k^k+2}{2}=1+2zT_k^k=2T_k-1=t_k\ ,\\
(z+4)\circ(4zT_k^k)&=4+4zT_k^k=4T_k\ .
\end{split}
\]
For $k=2l$ and $h=zt_{2l}^{2l-1}$, Corollary \ref{cor:aeration_bf_d} implies that
\[
B_h=DP_{l-1}(zD^2)=2T_{2l}^l P_{l-1}(4zT_{2l}^{2l})=2T_{2l}^l \cdot (U_{l-1}+U_{l-2})\circ(t_{2l}).
\]
For $k=2l+1$ and $h=zt_{2l+1}^{2l}$, Theorem \ref{thm:bh-d} implies that
\[
\begin{split}
B_h&=D\sqrt{p_{d-2}(zD^2)}=2T_{2l+1}^{l+\frac{1}{2}}\sqrt{p_{2l-1}(4zT_{2l+1}^{2l+1})}
=2T_{2l+1}^{l+\frac{1}{2}}\sqrt{4+4zT_{2l+1}^{2l+1}}\ U_{l-1}(t_{2l+1})\\
&=2T_{2l+1}^{l+\frac{1}{2}}\cdot 2T_{2l+1}^{\frac{1}{2}}U_{l-1}(t_{2l+1})
=4T_{2l+1}^{l+1}U_{l-1}(t_{2l+1}). \qedhere
\end{split}
\]
\end{proof}

We can express $B_h$ for $h=zt_k^{k-1}$ as a polynomial in $T_k$ even more simply.

\begin{theorem} \label{thm:b-double-catalan}
For $h=zt_k^{k-1}$, $k\ge 1$, and $\left[c_{k,n}\right]_{k,n\ge 0}$ as in \eqref{eq:cheb2}, we have 
\begin{equation} \label{eq:b-double-catalan}
B_h=2T_k^{k/2}\,U_{k-2}\!\left(\!\sqrt{T_k}\right)=\sum_{n=0}^{k-2}{2^{n+1} c_{k-2,n}T_k^{\frac{k+n}{2}}}.
\end{equation}
\end{theorem}

\begin{proof}
We have $B_h=\phi_k(T_k)$, where
\[
\phi_k(z)=
\begin{cases}
2z^l \cdot ((U_{l-1}+U_{l-2})\circ(2z-1))=2z^l (U_{l-1}^*+U_{l-2}^*)\ , & \text{if} \ k=2l;\\
4z^{l+1}U_{l-1}(2z-1)=4z^{l+1}U_{l-1}^*\ , & \text{if} \ k=2l+1.
\end{cases}
\]
The polynomial $U_l^*(z)=U_l(2z-1)$ is called the \emph{shifted} Chebyshev polynomial of the second kind and has the following properties (derivable from the definition of $U_k$ by replacing $\theta$ with $2\theta$ and using applicable trigonometric identities):
\[
\begin{split}
U_l^*(z^2)&=\frac{U_{2l+1}(z)}{2z}\ ,\\
U_l^*(z^2)+U_{l-1}^*(z^2)&=\frac{U_{2l+1}(z)+U_{2l-1}(z)}{2z}=U_{2l}(z).
\end{split}
\]
This implies that, for both $k=2l$ and $k=2l+1$, we have
\[
\phi_k(z^2)=2z^k U_{k-2}(z),
\]
or, equivalently, $\phi_k(z)=2z^{k/2}U_{k-2}(\sqrt{z})$. Applying Lemma \ref{lem:b-double-root} now yields the theorem.
\end{proof}

Note that $\frac{k+n}{2}=\frac{(k-2)+n}{2}+1$ in Equation \eqref{eq:b-double-catalan}, so all nonzero summands in the above sum contain only integer powers of $T_k$. Note also that for $k=1$, we have $t_1=\frac{1+z}{1-z}$, whose Riordan companion is $zt_1^0=z$ with $B_z=0$, which agrees with Theorem \ref{thm:b-double-catalan}, since the sum in the theorem statement is empty for $k=1$.
\begin{example} \label{ex:b-double-catalan}
The matrix $\left[c_{k,n}\right]_{k,n\ge 0}$ begins
\[
\begin{bmatrix*}[r]
1 & 0 & 0 & 0 & 0 & 0 & 0\\
0 & 1 & 0 & 0 & 0 & 0 & 0\\
-1 & 0 & 1 & 0 & 0 & 0 & 0\\
0 & -2 & 0 & 1 & 0 & 0 & 0\\
1 & 0 & -3 & 0 & 1 & 0 & 0\\
0 & 3 & 0 & -4 & 0 & 1 & 0\\
-1 & 0 & 6 & 0 & -5 & 0 & 1
\end{bmatrix*},
\]
so the matrix $\left[2^{n+1}c_{k,n}\right]_{k,n\ge 0}$ begins
\[
\begin{bmatrix*}[r]
2 & 0 & 0 & 0 & 0 & 0 & 0\\
0 & 4 & 0 & 0 & 0 & 0 & 0\\
-2 & 0 & 8 & 0 & 0 & 0 & 0\\
0 & -8 & 0 & 16 & 0 & 0 & 0\\
2 & 0 & -24 & 0 & 32 & 0 & 0\\
0 & 12 & 0 & -64 & 0 & 64 & 0\\
-2 & 0 & 48 & 0 & -160 & 0 & 128
\end{bmatrix*}.
\]
Hence, for the first few values of $k$, we have
\[
\begin{split}
B_{zt_2}&=2T_2=2C,\\
B_{zt_3^2}&=4T_3^2,\\
B_{zt_4^3}&=-2T_4^2+8T_4^3=2T_4^2\cdot(2t_4+1),\\
B_{zt_5^4}&=-8T_5^3+16T_5^4=4T_5^3\cdot2t_5,\\
B_{zt_6^5}&=2T_6^3-24T_6^4+32T_6^5=2T_6^3\cdot(4t_6^2+2t_6-1).
\end{split}
\]
\end{example}

\subsection{Generalized Schr\"{o}der trees} \label{subsec:gen-schroeder}

Another family we can analyze similarly to the $k$-ary trees and double $k$-ary trees is a special case of Example \ref{ex:schroeder-tree}: a family $r_k=r_k(z)$, $k\ge 1$, that satisfies the equation 
\[
r_k=1+z(r_k^{k-1}+r_k^k).
\]
Then $r_1=\frac{1+z}{1-z}=t_1$ and $r_2=r$, the generating function for the large Schr\"{o}der numbers. Note that in this case, $\gamma=z^{k-1}+z^k$ in Equation \eqref{eq:gamma-ogf}, so $d=\dar(\gamma)=(k-1)+k=2k-1$, and thus the companion of $r_k$ is $zr_k^{2k-2}$, i.e., $(r_k,zr_k^{2k-2})$ is a pseudo-involution in the $(2k-2)$-Bell subgroup.

Consider the $\delta$ and $D$ functions for $r_k$. From Equations \eqref{eq:def-delta} and \eqref{eq:def-d}, we have
\[
\delta^2\left(\frac{(z-1)^2}{z}\right)=\frac{\gamma^2(z)}{z^d}=\frac{(1+z)^2 z^{2k-2}}{z^{2k-1}}=\frac{(1+z)^2}{z}=4+\frac{(z-1)^2}{z},
\]
so $\delta^2(z)=z+4$, and thus
\[
zD^2=\overline{\left(\frac{z}{z+4}\right)}=\frac{4z}{1-z} \qquad \text{and} \qquad D=\frac{2}{\sqrt{1-z}}\ .
\]

Therefore, applying Corollary \ref{cor:aeration_bf_d} to $h=zr_k^{2k-2}$ and $d=2k-1$, we get
\[
\begin{split}
B_h &= D\sqrt{p_{2k-3}(zD^2)} = D\sqrt{zD^2+4}\; U_{k-2}\!\left(\frac{zD^2+2}{2}\right)\\
&=\frac{2}{\sqrt{1-z}}\cdot\frac{2}{\sqrt{1-z}}\cdot U_{k-2}\!\left(\frac{1+z}{1-z}\right)=\frac{4}{1-z}U_{k-2}\!\left(\frac{1+z}{1-z}\right)
\end{split}
\]

For example, since $U_0(z)=1$, $U_1(z)=2z$, and $U_2(z)=4z^2-1$, it follows that
\[
\begin{split}
B_{zr^2}=B_{zr_2^2}&=\frac{4}{1-z} \ ,\\
B_{zr_3^4}&=\frac{4}{1-z}U_1\left(\frac{1+z}{1-z}\right) =\frac{8(1+z)}{(1-z)^2} \ ,\\
B_{zr_4^6}&=\frac{4}{1-z}U_2\left(\frac{1+z}{1-z}\right)=\frac{4(3+z)(1+3z)}{(1-z)^3} \ ,
\end{split}
\]
and so on. Note that, just like for the previous family,  we have $r_1=\frac{1+z}{1-z}$ for $k=1$, with the Riordan companion of $h=zr_1^0=z$ and $B_h=B_z=0$.

\subsection{Increasing trees of even arity} \label{subsec:inc-tree}

To conclude this section, we want to give an example of functions $f(z)$ and $h(z)$ as in \eqref{eq:fh-def}, where $B_h$ is easier to find, which in turn allows us to find $B_f$ using Equation \eqref{eq:bf-bh} of Theorem \ref{thm:aeration}.

A $(d+1)$-ary \emph{increasing tree} is a rooted, partially labeled tree in which each internal vertex has outdegree (arity) $d+1$, leaves are unlabeled, internal vertices are labeled with consecutive integers starting with $1$, and labels increase along any path away from the root. A $(d+1)$-ary tree with $n$ internal vertices has $dn+1$ leaves and $(d+1)n+1$ vertices in total. For an extensive treatment of various types of increasing trees, see, e.g., Bergeron et al.\ \cite{BFS}. We also mention that $(d+1)$-ary increasing trees are in bijection with \emph{$d$-Stirling permutations} \cite{JKP}, i.e., permutations of $\{1^d,2^d,\dots,n^d\}$ where any subsequence $a\cdots b\cdots a$ satisfies $b\ge a$. From \cite[Table 1]{BFS}, we see that the exponential generating function for the number of such trees with $n$ internal vertices (including the singleton tree when $n=0$) is given by 
$$
\frac{1}{(1-dz)^{\frac{1}{d}}}.
$$
We will instead consider these trees according to the number of leaves, i.e., the weight of a tree with $n$ internal vertices will be $z^{dn+1}$. This yields the exponential generating function
\[
\frac{z}{(1-dz^d)^{\frac{1}{d}}}.
\]
To generalize the class of functions under consideration slightly, we let
\[
f=f(z)=\frac{z}{(1-cz^d)^{\frac{1}{d}}},
\]
where $c$ is an arbitrary constant. For a nonnegative integer $c$, this corresponds to allowing $c$ possible colors for each internal vertex, while leaves remain uncolored. Note that this implies that
\[
z^d\circ f = f^d = \frac{z^d}{1-cz^d} = \frac{z}{1-cz} \circ z^d,
\]
and since $z/(1-cz)$ is pseudo-involutory, it follows that so is $f$ when $d$ is odd (i.e., when the outdegree $d+1$ is even). In that case, let $l=(d-1)/2$, so that $d=2l+1$.

Note that
\[
f(z)=\frac{z}{(1-cz^d)^{\frac{1}{d}}}=\sum_{n=0}^{\infty}\left(\left(\frac{c}{d}\right)^n\prod_{j=0}^{n}(jd+1)\right)\frac{x^n}{n!},
\]
so the coefficient at $x^n/n!$ is an integer if $d\mid c$, and the coefficient at $x^n$ is an integer if $d^2\mid c$. When $d=1$, this yields the coefficient $n!$ at $x^n/n!$, corresponding to the well-known fact that binary increasing trees are in bijection with permutations. The expression $\prod_{j=0}^{n}(jd+1)$ is sometimes denoted $(dn+1)(!^d)$, i.e., $n!$, $(2n+1)!!$, $(3n+1)!!!$, etc. 

Let $g=g(z)=1/(1-cz)^{1/d}$, then $f(z)=zg(z^d)$ as in \eqref{eq:bf-bh}, while the corresponding function $h$ from \eqref{eq:bf-bh} is $h=zg^d=z/(1-cz)$. Therefore, $B_h(z)=c$ is constant, so $zB_h^2=c^2z$, and thus \eqref{eq:bf-bh} implies that
\[
(zP_l^2)\circ(zB_f^2)=(zB_h^2)\circ(z^{2l+1})=(c^2z)\circ(z^{2l+1})=c^2z^{2l+1},
\]
or, equivalently,
\[
zB_f^2=\overline{(zP_l^2)}\circ(c^2z^{2l+1}),
\]
where $P_l=P_l(z)$ is as in \eqref{eq:pl-def}.

\begin{example}
Letting $d=1$ (i.e., considering binary increasing trees) yields $l=0$ and $P_l=1$, so that $B_f=cz^l=c$. Letting $d=3$ (i.e., considering quaternary increasing trees) yields $l=1$ and $P_l=3+z$, so for the corresponding function $f=f(z)=z/(1-cz^3)^{1/3}$ (see \seqnum{A007559} for $c=3$), we have
\[
zB_f^2=\overline{(z(3+z)^2)}\circ(c^2z^3).
\]
Note that
\[
\overline{z(3+z)^2}=\overline{(27z)\circ(z(1+z)^2)\circ\left(\frac{z}{3}\right)}=(3z)\circ\left(zT_3^2(-z)\right)\circ\left(\frac{z}{27}\right)=\frac{z}{9}\,T_3^2\left(-\frac{z}{27}\right),
\]
and therefore
\[
zB_f^2=\frac{c^2z^3}{9}\,T_3^2\left(-\frac{c^2z^3}{27}\right),
\]
or, equivalently,
\[
B_f=\frac{cz}{3}\,T_3\left(-\frac{c^2z^3}{27}\right).
\]
In particular, this yields $B_f=zT_3(-z^3/3)$ for $c=3$ and $B_f=3zT_3(-3z^3)$ for $c=9$.
\end{example}

\section{Pseudo-conjugation of pseudo-involutions} \label{sec:b-pal}

In this section, we will give a simple general result useful in obtaining  new pseudo-involutions from the existing ones by compositions with invertible functions. This will be followed by several applications of that general result. Furthermore, we will find the $B$-functions for some of the families of pseudo-involutions in our examples.

\begin{definition} \label{def:pseudo-inverse}
Let $h$ be an invertible formal power series, and define the \emph{pseudo-inverse} $\widehat{h}$ of $h$ to be $\widehat{h}=(-z)\circ \overline{h}\circ (-z)$.
\end{definition}

Note that  $\widehat{h}=h$ if and only if $\overline{h}=(-z)\circ h \circ (-z)$, i.e., if and only if $h$ is pseudo-involutory. Note also that, for an arbitrary invertible $h$, we have $\overline{\widehat{h}}=\widehat{\overline{h}}=(-z)\circ h\circ (-z)$.

\begin{theorem} \label{thm:inverse-prop}
Let $h$ be an invertible formal power series. Then $(g,f)$ is a pseudo-involution if and only if $(g\circ h, \widehat{h}\circ f\circ h)$ is a pseudo-involution.
\end{theorem}

Note that $(g\circ h, \widehat{h}\circ f\circ h)=(1,h)(g,f)(1,\widehat{h})$, with the order of $h$ and $\widehat{h}$ reversed. We call $\widehat{h}\circ f\circ h$ the \emph{pseudo-conjugate} of $f$ by $h$. Likewise, $(g\circ h, \widehat{h}\circ f\circ h)$ is the \emph{pseudo-conjugate} of $(g,f)$ by $(1,h)$. In each case, we call the corresponding operation \emph{pseudo-conjugation}.

\begin{corollary} \label{cor:inverse-prop}
Let $h$ be pseudo-involutory. Then $(g,f)$ is a pseudo-involution if and only if $(g\circ h, h\circ f\circ h)$ is a pseudo-involution.
\end{corollary}

The same statements also hold for the exponential pseudo-involutions $[g,f]$ and $[g\circ h, \widehat{h}\circ f\circ h]$.

We note that Theorem \ref{thm:inverse-prop} and Corollary \ref{cor:inverse-prop} have already appeared in He and Shapiro \cite{HSh2} as Theorem 5.1 and Corollary 5.2, as part of the ``very helpful comments'' by the first author of this paper, referred to in  \cite[Acknowledgments]{HSh2}.

\begin{proof}[Proof of Theorem \ref{thm:inverse-prop}]
$(g,f)$ is a pseudo-involution if and only if $g(-f)=1/g$, i.e.
\[
g\circ (-z) \circ f = \frac{1}{g}.
\]
But then
\[
\begin{split}
(g\circ h)\circ (-z)\circ (\widehat{h}\circ f\circ h) &= g\circ h\circ (-z)\circ (-z)\circ \overline{h} \circ (-z) \circ f\circ h \\&= g\circ (-z)\circ f\circ h = \frac{1}{g}\circ h = \frac{1}{g\circ h},
\end{split}
\]
i.e., $(g\circ h, \widehat{h}\circ f\circ h)$ is a pseudo-involution.

The converse follows easily from the above and the facts that $g=(g\circ h)\circ \overline{h}$ and $\overline{\widehat{h}}=\widehat{\overline{h}}$.
\end{proof}

We note, following He and Shapiro \cite{HSh2}, that when $h$ is pseudo-involutory, the Riordan array
\[
(g\circ h, h\circ f\circ h)=(1,h)(g,f)(1,h)
\]
is a pseudo-involution as a palindromic product of pseudo-involutions.

We will first give a few examples of Corollary \ref{cor:inverse-prop}, where $h$ is pseudo-involutory.

\begin{example} \label{ex:hfh_zc3}
Let $h=\dfrac{z}{1-z}$ and $g=\tilde{m}$ of Example \ref{ex:motzkin-tree}. Then $f=z\tilde{m}$, so $g\circ h=C$, and
\[
h\circ f\circ h = \frac{z}{1-z}\circ (z\tilde{m}) \circ \frac{z}{1-z} 
= \frac{z}{1-z} \circ \frac{zC}{1-z} = \frac{zC}{1-z-zC}=zC^3.
\] 
This uses the identity $\tilde{m}\left(\frac{z}{1-z}\right)=C$, or equivalently, $\tilde{m}=C\left(\frac{z}{1+z}\right)$. The last of the chain of equalities is obtained by multiplying through by $C$ and using the identity $C-zC^2=1$ twice. Thus, we (re-)obtain pseudo-involution $(C,zC^3)$ from $(\tilde{m},z\tilde{m})$.
\end{example}

\begin{example} \label{ex:hfh_fm}
Let $h=z\tilde{m}$ and $g=\dfrac{1}{1-z}$. Then $f=\dfrac{z}{1-z}$, so $g\circ h=\dfrac{1}{1-z\tilde{m}}=m$, and
\[
\begin{split}
h\circ f\circ h&=(z\tilde{m})\circ\frac{z}{1-z}\circ(z\tilde{m}) = \frac{zC}{1-z}\circ(z\tilde{m})\\
&=\frac{z\tilde{m}}{1-z\tilde{m}}C(z\tilde{m})=z\tilde{m}m C(z\tilde{m})=(m-1)C(z\tilde{m}).
\end{split}
\]
Thus, we (re-)obtain the pseudo-involution $\left(m,(m-1)C(z\tilde{m})\right)$ and see that the pseudo-involutory companion to the Motzkin generating function $m$ is
\[
(m-1)C(z\tilde{m})=m\cdot((zC)\circ(z\tilde{m})).
\]
\end{example}

\begin{example} \label{ex:hfh_2m-1}
Let $h=z\tilde{m}$ and $g=\frac{1+z}{1-z}$. Then $f=z$, so $g\circ h=\frac{1+z\tilde{m}}{1-z\tilde{m}}=m(1+z\tilde{m})=m+(m-1)=2m-1$, and $h\circ f\circ h=(z\tilde{m})\circ(z\tilde{m})$. This yields the pseudo-involution
\[
\left(2m-1,(z\tilde{m})\circ(z\tilde{m})\right).
\]
Similarly, we can see that $\left(\frac{1+h}{1-h}\ , \widehat{h}\circ h\right)$ is a pseudo-involution for any invertible function $h$.

Note that the coefficient sequence \seqnum{A348197} of $(z\tilde{m})\circ(z\tilde{m})$ begins
\[
0,1,2,4,10,28,84,264,860,2880,9862,34392,\dots,
\]
and, curiously, its first 8 terms coincide with those of $z(2C-1)$. From the 9th term on, it appears that $\left((z\tilde{m})\circ(z\tilde{m})\right)-z(2C-1)=2z^8+20z^9+138z^{10}+800z^{11}+\cdots$ has only positive coefficients. It would be interesting to see a combinatorial interpretation of that fact.
\end{example}

\begin{example} \label{ex:hfh_2tm-1}
Let $h=\dfrac{z}{1+z}$ and $g=2C-1$. Then $f=z(2C-1)$, so $g\circ h=(2C-1)\circ\dfrac{z}{1+z}=2\tilde{m}-1$, and
\[
\begin{split}
h\circ f\circ h&=\frac{z}{1+z}\circ (z(2C-1))\circ\frac{z}{1+z}\\
&=\frac{z}{1+z}\circ\frac{z(2\tilde{m}-1)}{1+z}=\frac{z(2\tilde{m}-1)}{1+z+z(2\tilde{m}-1)}=\frac{z(2\tilde{m}-1)}{1+2z\tilde{m}}.
\end{split}
\]
Thus,
\[
\left(2\tilde{m}-1,\frac{z(2\tilde{m}-1)}{1+2z\tilde{m}}\right)
\]
is a pseudo-involution. The pseudo-involutory companion
\[
\frac{z(2\tilde{m}-1)}{1+2z\tilde{m}}=\frac{1-\sqrt{1-2z-3z^2}}{2+z-\sqrt{1-2z-3z^2}}
\]
of $2\tilde{m}-1$ corresponds to the coefficient sequence \seqnum{A348189} that begins
\[
0,1,0,0,2,0,6,8,24,60,148,396,1026,2744,7350,19872,54102,148104,407682,\dots\ .
\]
It would be interesting to see a natural combinatorial interpretation of this sequence.
\end{example}

Now we give some examples of the application of Theorem \ref{thm:inverse-prop}, where $h$ is not necessarily pseudo-involutory. An easy example involves the Fibonacci generating function.
\begin{example} \label{ex:hfh_fib}
Let $g=\frac{1}{1-z}$ and $h=z+z^2$, then $f=\frac{z}{1-z}$\ , $\widehat{h}=\overline{(z-z^2)}=zC$, $F=g\circ h=\frac{1}{1-z-z^2}$ is the generating function for the Fibonacci numbers, and
\[
\widehat{h}\circ f\circ h = (zC) \circ \frac{z}{1-z} \circ (z+z^2)= (zC) \circ (F-1)=(F-1)\cdot C(F-1).
\]
Thus, $(F,(zC)\circ(F-1))$ is a pseudo-involution. It begins
\[ 
\begin{bmatrix}
1 & 0 & 0 & 0 & 0 & 0 \\ 
1 & 1 & 0 & 0 & 0 & 0 \\ 
2 & 4 & 1 & 0 & 0 & 0 \\ 
3 & 14 & 7 & 1 & 0 & 0 \\ 
5 & 50 & 35 & 10 & 1 & 0 \\ 
8 & 190 & 160 & 65 & 13 & 1
\end{bmatrix}.
\]
We note that the pseudo-involutory companion
\[
(zC)\circ(F-1)=\frac{1}{2}\left(1-\sqrt{\frac{1-5z-5z^2}{1-z-z^2}}\right)
\]
is the generating function for the sequence \seqnum{A344623}. The methods for finding $B$-functions discussed in this paper do not appear to apply in this case, but comparing a few initial terms of this pseudo-involution's $B$-sequence with existing OEIS sequences allows us to conjecture and verify the entire $B$-sequence: it is, in fact, \seqnum{A200031} with the initial $1$ replaced by $3$. Thus, $y=B_{(zC)\circ(F-1)}$ satisfies the functional equation
\[
y=3-z-zy+zy^2
\]
and is given explicitly by
\[
B_{(zC)\circ(F-1)}=\frac{1+z-\sqrt{1-10z+5z^2}}{2z}.
\]
\end{example}

\bigskip

The next pair of examples involves two related families of functions (so related that we will see that it is, in fact, a single family of functions), with both aerated and non-aerated versions of those families. These examples rely on the following identity satisfied by the Catalan generating function $C=1+zC^2$:
\begin{equation} \label{eq:cat-id}
1+zC^3=1+(C-1)C=1-C+C^2=C^2-zC^2=(1-z)C^2,
\end{equation}
or, equivalently,
\begin{equation} \label{eq:cat-id-comp}
1-z=\frac{1-C+C^2}{C^2}=\frac{1-z+z^2}{z^2}\circ C=\frac{1+z+z^2}{(1+z)^2}\circ(zC^2).
\end{equation}

\begin{example} \label{ex:hfh-ctpos}
Let $T_n=T_n(z)$ be the generating function for the $n$-ary trees from Example \ref{ex:catalan-tree}, i.e., $T_n=1+zT_n^n$. Let $g=C$ and $h=zT_n^{n-1}$, then $f=zC^3$ and $\overline{h}=z(1-z)^{n-1}$, so $\widehat{h}=z(1+z)^{n-1}$, and hence
\[
\widehat{h}\circ f\circ h=(z(1+z)^{n-1})\circ(zC^3)\circ(zT_n^{n-1}).
\]
However, from Equation \eqref{eq:cat-id},
\[
(z(1+z)^{n-1})\circ(zC^3)=zC^3(1+zC^3)^{n-1}=zC^3\left((1-z)C^2\right)^{n-1}=z(1-z)^{n-1}C^{2n+1},
\]
so, recalling that $z(1-z)^{n-1}=\overline{zT_n^{n-1}}$, we have
\[
\widehat{h}\circ f\circ h=\left(z(1-z)^{n-1}C^{2n+1}\right)\circ\left(zT_n^{n-1}\right)
=zC^{2n+1}(zT_n^{n-1}).
\]
Thus, $(C(zT_n^{n-1}),zC^{2n+1}(zT_n^{n-1}))$ is a pseudo-involution in the $(2n+1)$-Bell subgroup.

Note also that $T_n=1/(1-zT_n^{n-1})$, and $C=\tilde{m}\left(\frac{z}{1-z}\right)$, so
\[
C(zT_n^{n-1})=\tilde{m}\left(\frac{zT_n^{n-1}}{1-zT_n^{n-1}}\right)=\tilde{m}(zT_n^n),
\]
and thus
\[
(C(zT_n^{n-1}),zC^{2n+1}(zT_n^{n-1}))=(\tilde{m}(zT_n^n),z\tilde{m}^{2n+1}(zT_n^n)).
\]

\medskip

Thus, letting $n=0$, we recover the pseudo-involution $(\tilde{m},z\tilde{m})$; letting $n=1$, we recover the pseudo-involution $(C,zC^3)$.

Letting $n=2$ in this example yields the pseudo-involution $(C(zC),zC^5(zC))$. The function $C(zC)$ is the generating function for the sequence \seqnum{A127632}, which is the counting sequence e.g., for permutations that avoid the set of patterns $(1342,3142,45132)$ and do not start with a descent. Alternatively, as shown by Archer and Graves \cite{AG}, $C(zC)$ is the generating function for the number of 132-avoiding permutations of length $3n$, $n\ge 0$, whose disjoint cycle decomposition contains only 3-cycles $(a,b,c)$ with $a>b>c$.

Likewise, letting $n=3$ and $n=4$, we obtain pseudo-involutions $(C(zT_3^2),zC^7(zT_3^2))$ and $(C(zT_4^3),zC^9(zT_4^3))$, where the first components are the generating functions for sequences \seqnum{A153295} and \seqnum{A153396}, respectively.
\end{example}

Finally, for comparison, note that, for $n\ge 2$, we have $T_n=1+(zT_n^{n-2})T_n^2$, so $T_n=C(zT_n^{n-2})$, and thus $(C(zT_n^{n-2}),zC^{2n-1}(zT_n^{n-2}))=(T_n,zT_n^{2n-1})$ is also a pseudo-involution (see Example \ref{ex:catalan-tree}).

\begin{example} \label{ex:hfh-ctneg}
The following family is complementary to the one in Example \ref{ex:hfh-ctpos}. With $T_n$ as before, let $g=1/C=1-zC$ and $h=-zT_n^n$, then $f=zC^3$ and $\overline{h}=-z/(1-z)^n$, so $\widehat{h}=-z/(1+z)^n$, and hence
\[
\widehat{h}\circ f\circ h=(-z/(1+z)^n)\circ(zC^3)\circ(-zT_n^n).
\] 
However,
\[
\left(-\frac{z}{(1+z)^n}\right)\circ(zC^3)=-\frac{zC^3}{(1+zC^3)^n}=-\frac{zC^3}{((1-z)C^2)^n}=-\frac{z}{(1-z)^n C^{2n-3}},
\]
so, recalling that $-z/(1-z)^n=\overline{(-zT_n^n)}$, we have
\[
\widehat{h}\circ f\circ h=\left(-\frac{z}{(1-z)^n C^{2n-3}}\right)\circ(-zT_n^n)=\frac{z}{C^{2n-3}(-zT_n^n)}.
\]
Thus,
\[
\left(\frac{1}{C(-zT_n^n)},\frac{z}{C^{2n-3}(-zT_n^n)}\right)
\]
is a pseudo-involution in the $(2n-3)$-Bell subgroup. In particular, for $n=2$, the pseudo-involution
\[
\left(\frac{1}{C(-zC^2)},\frac{z}{C(-zC^2)}\right)
\]
is in the Bell subgroup.

Note that, despite the negative signs and division in the generating function  $1/C(-zT_n^n)$, its sequence of coefficients consists of positive integers only. For example, for $n=2$, the coefficient sequence \seqnum{A166135} of $1/C(-zC^2)$ begins
\[
1,1,1,3,7,22,65,213,693,2352,8034,\dots,
\]
and for $n=3$, the coefficient sequence \seqnum{A347953} of $1/C(-zT_3^3)$ begins
\[
1,1,2,8,35,171,882,4744,26286,149045,860596,\dots\ .
\]
The pseudo-involution $\left(\frac{1}{C(-zC^2)},\frac{z}{C(-zC^2)}\right)$ begins
\[
\begin{bmatrix}
1 & 0 & 0 & 0 & 0 & 0 & 0 & 0 & 0\\
1 & 1 & 0 & 0 & 0 & 0 & 0 & 0 & 0\\
1 & 2 & 1 & 0 & 0 & 0 & 0 & 0 & 0\\
3 & 3 & 3 & 1 & 0 & 0 & 0 & 0 & 0\\
7 & 8 & 6 & 4 & 1 & 0 & 0 & 0 & 0\\
22 & 21 & 16 & 10 & 5 & 1 & 0 & 0 & 0\\
65 & 64 & 45 & 28 & 15 & 6 & 1 & 0 & 0\\
213 & 197 & 138 & 83 & 45 & 21 & 7 & 1 & 0\\
693 & 642 & 436 & 260 & 140 & 68 & 28 & 8 & 1
\end{bmatrix}.
\]

The sequence $1/C(-zC^2)$ is the generating function for the \emph{basketball walks} defined in Ayyer and Zeilberger \cite{AZ} and considered using the kernel method in Banderier et al.\ \cite{BKKKKNW} and bijectively in Bettinelli et al.\ \cite{BFMR}, i.e., lattice paths starting at $(0,0)$ that either consist of a single point, or have only steps from among $(1,1)$, $(1,-1)$, $(1,2)$, $(1,-2)$, stay at positive height except at $(0,0)$, and end at height $1$.

Unfortunately, we were unable to construct a lattice path generalization of basketball walks whose counting sequence would have $1/C(-zT_n^n)$ as its generating function. Such a generalization or another natural combinatorial interpretation of the coefficients of $1/C(-zT_n^n)$ would be interesting to see.

Finally, notice that $T_n=1+zT_n^n$ and $C=\tilde{m}\left(\frac{z}{1-z}\right)$ imply that
\[
C(-zT_n^n)=\tilde{m}\left(\frac{-zT_n^n}{1+zT_n^n}\right)=\tilde{m}(-zT_n^{n-1}),
\]
and thus
\[
\left(\frac{1}{C(-zT_n^n)},\frac{z}{C^{2n-3}(-zT_n^n)}\right)=\left(\frac{1}{\tilde{m}(-zT_n^{n-1})},\frac{z}{\tilde{m}^{2n-3}(-zT_n^{n-1})}\right).
\]
\end{example}

The two families of functions in Examples \ref{ex:hfh-ctpos} and \ref{ex:hfh-ctneg} are complementary in the following way. We can extend the family of $n$-ary tree generating functions $T_n$ from $n\ge 1$ to all $n\in\mathbb{Z}$ using the functional equation $T_n=1+zT_n^n$. Then, for $n\ge 0$, we have $T_{-n}=1+zT_{-n}^{-n}$. Dividing this through by $T_{-n}$ and rearranging terms, we obtain
\[
\frac{1}{T_{-n}}=1-z\left(\frac{1}{T_{-n}}\right)^{n+1},
\]
i.e.
\[
\frac{1}{T_{-n}}=T_{n+1}(-z),
\]
or, equivalently,
\begin{equation} \label{eq:cat-neg}
T_{-n}=\frac{1}{T_{n+1}(-z)}.
\end{equation}
Now denote the pseudo-involutory companions in Examples \ref{ex:hfh-ctpos} and \ref{ex:hfh-ctneg} by
\begin{equation} \label{eq:def-h12}
h_{1,n}=zC^{2n+1}(zT_n^{n-1}), \qquad h_{2,n}=\frac{z}{C^{2n-3}(-zT_n^n)}\ .
\end{equation}
Then
\[
\frac{z}{C^{2(-n)-3}(-zT_{-n}^{-n})}=\frac{z}{C^{-2n-3}(-zT_{n+1}^{n}(-z))}=(-z)\circ (zC^{2n+3}(zT_{n+1}^{n}))\circ(-z),
\]
or, in other words,
\[
h_{2,-n}=(-z)\circ h_{1,n+1}\circ (-z)=\overline{h_{1,n+1}},
\]
and therefore, similarly,
\[
h_{1,-n}=(-z)\circ h_{2,n+1}\circ (-z)=\overline{h_{2,n+1}}.
\]

As an alternative way to show that functions $h_{1,n}$ and $h_{2,n}$ are pseudo-involutory and yield pseudo-involutions in the $k$-Bell subgroups for $k=2n+1$ and $k=2n-3$, respectively, we can produce the corresponding generalized palindromes $\gamma$ 
satisfying \eqref{eq:gamma-ogf} for $g=C(zT_n^{n-1})$ and $g=1/C(-zT_n^n)$, respectively. 

For $g=C(zT_n^{n-1})$, we have
\[
g-1=C(zT_n^{n-1})-1=(zC^2)\circ(zT_n^{n-1}),
\]
which implies
\[
\begin{split}
z&=\overline{(zT_n^{n-1})}\circ\overline{(zC^2)}\circ(g-1)
=\left(z(1-z)^{n-1}\right)\circ\left(\frac{z}{(1+z)^2}\right)\circ(g-1)\\
&=\frac{z(1+z+z^2)^{n-1}}{(1+z)^{2n}}\circ(g-1)=\frac{(g-1)(1-g+g^2)^{n-1}}{g^{2n}},
\end{split}
\]
and therefore,
\[
g=1+z\,\frac{g^{2n}}{(1-g+g^2)^{n-1}}\ ,
\]
so
\begin{equation} \label{eq:hfh-ctpos_gamma}
\gamma(z)=\frac{z^{2n}}{(1-z+z^2)^{n-1}}\ ,
\end{equation}
which is palindromic with $\dar(\gamma)=(2n+2n)-(0+2)\cdot(n-1)=2n+2$, and thus, by Theorem~\ref{thm:palindrome}, $(g,zg^{2n+1})$ is a pseudo-involution.

Similarly, for $g=1/C(-zT_n^n)$, we have
\[
g-1=\frac{1}{C(-zT_n^n)}-1=(-zC)\circ(-zT_n^n)=(-z)\circ(zC)\circ(-z)\circ(zT_n^n),
\]
and therefore
\[
\begin{split}
z&=\overline{(zT_n^n)}\circ(-z)\circ\overline{(zC)}\circ(-z)\circ(g-1)
=\frac{z}{(1+z)^n}\circ(-z)\circ(z-z^2)\circ(-z)\circ(g-1)\\
&=\frac{z}{(1+z)^n}\circ(z+z^2)\circ(g-1)=\frac{z(1+z)}{(1+z+z^2)^n}\circ(g-1)=\frac{(g-1)g}{(1-g+g^2)^n},
\end{split}
\]
which implies
\[
g=1+z\,\frac{(1-g+g^2)^n}{g}\ ,
\]
so
\begin{equation} \label{eq:hfh-ctneg_gamma}
\gamma(z)=\frac{(1-z+z^2)^n}{z}\ ,
\end{equation}
which is palindromic with $\dar(\gamma)=(0+2)\cdot n - (1+1) = 2n-2$, and thus, by Theorem~\ref{thm:palindrome}, $(g,zg^{2n-3})$ is a pseudo-involution.

\bigskip

Since both powers $2n+1$ and $2n-3$ are odd, it follows that the aerations
\[
\left(C(z^{2n+1}T_n^{n-1}(z^{2n+1}),zC(z^{2n+1}T_n^{n-1}(z^{2n+1})\right)
\]
and
\[
\left(\frac{1}{C(-z^{2n-3}T_n^n(-z^{2n-3}))},\frac{z}{C(-z^{2n-3}T_n^n(-z^{2n-3}))}\right)
\]
are also pseudo-involutions in the Bell subgroup. We will denote the pseudo-involutory companions in those pseudo-involutions as
\begin{equation} \label{eq:def-f12}
f_{1,n}=zC(z^{2n+1}T_n^{n-1}(z^{2n+1})), \qquad f_{2,n}=\frac{z}{C(-z^{2n-3}T_n^n(-z^{2n-3}))}\ .
\end{equation}

As we have seen in Example \ref{ex:b-cat} and Theorem \ref{thm:b-aerated-double-root}, it is often the case that the $B$-function for a $k$-Bell pseudo-involution $(g,zg^k)$ is difficult to find directly, whereas the $B$-function for the associated aerated pseudo-involution $(g(z^k),zg(z^k))$ in the Bell subgroup is somewhat easier to obtain. Such is the case with the families in Examples \ref{ex:hfh-ctpos} and \ref{ex:hfh-ctneg}.

\begin{theorem} \label{thm:hfh_bf_bh}
The $B$-functions for the pseudo-involutory functions $f_{1,n}$ and $f_{2,n}$ defined in \eqref{eq:def-f12} and $h_{1,n}$ and $h_{2,n}$ defined in \eqref{eq:def-h12} are as follows:
\begin{itemize}
\item for $f_{1,n}=zC(z^{2n+1}T_n^{n-1}(z^{2n+1}))$ and $h_{1,n}=zC^{2n+1}(zT_n^{n-1})$, we have
\[
B_{f_{1,n}}=z^nT_{2n-1}^{n-1}(-z^{2n+1}) \qquad \text{and} \qquad B_{h_{1,n}}=T_{2n-1}^{n-1}(-z)\, P_n\!\left(zT_{2n-1}^{2n-2}(-z)\right),
\]
\item for $f_{2,n}=\dfrac{z}{C(-z^{2n-3}T_n^{n}(z^{2n-3}))}$ and $h_{2,n}=\dfrac{z}{C^{2n-3}(-zT_n^n)}$\ , we have
\[
B_{f_{2,n}}=z^{n-2}T_{2n}^{n}(z^{2n-3})  \qquad \text{and} \qquad B_{h_{2,n}}=T_{2n}^{n}\, P_{n-2}\!\left(zT_{2n}^{2n}\right).
\]
\end{itemize}
\end{theorem}

\begin{proof}
Let functions $\delta$ and $D$ be defined as in \eqref{eq:def-delta} and \eqref{eq:def-d}, respectively. From Corollary \ref{cor:aeration_bf_d}, it follows that we only need to find the function $D$ in each case. 

For the first family, with $\gamma$ as in \eqref{eq:hfh-ctpos_gamma} and $d=\dar(\gamma)=2n+2$, we have
\[
\delta^2\left(\frac{(z-1)^2}{z}\right)=\frac{\gamma(z)^2}{z^d}=\frac{1}{z^{2n+2}}\cdot\frac{z^{4n}}{(1-z+z^2)^{2n-2}}=\left(\frac{z}{1-z+z^2}\right)^{2n-2}=\frac{1}{\left(1+\frac{(z-1)^2}{z}\right)^{2n-2}}\ ,
\]
so $\delta^2(z)=1/(1+z)^{2n-2}$ and therefore $z/\delta^2=z(1+z)^{2n-2}$. This implies
\[
zD^2=\overline{z(1+z)^{2n-2}}=zT_{2n-1}^{2n-2}(-z),
\]
and therefore,
\[
D=T_{2n-1}^{n-1}(-z).
\]

Likewise, for the second family, with $\gamma$ as in \eqref{eq:hfh-ctneg_gamma} and $d=\dar(\gamma)=2n-2$, we have
\[
\delta^2\left(\frac{(z-1)^2}{z}\right)=\frac{\gamma(z)^2}{z^d}=\frac{1}{z^{2n-2}}\cdot\frac{(1-z+z^2)^{2n}}{z^2}=\left(\frac{1-z+z^2}{z}\right)^{2n}=\left(1+\frac{(z-1)^2}{z}\right)^{2n},
\]
so $\delta^2(z)=(1+z)^{2n}$ and therefore $z/\delta^2=z/(1+z)^{2n}$. This implies
\[
zD^2=\overline{\left(\frac{z}{(1+z)^{2n}}\right)}=zT_{2n}^{2n},
\]
and therefore,
\[
D=T_{2n}^{n}.
\]
Since in each case, the darga $d$ is even, and hence $d-1$ is odd, the theorem now follows from Theorem \ref{thm:palindrome}, Theorem \ref{thm:aeration}, and Corollary \ref{cor:aeration_bf_d}. 
\end{proof}

\begin{example} \label{ex:hfh_bf_123}
We will give a few initial examples of the results of Theorem \ref{thm:hfh_bf_bh}.

In the first family, letting $n=0$ yields $T_0=1+z$, $f=zC\left(\frac{z}{1+z}\right)=z\tilde{m}$, and $B_f=T_{-1}^{-1}(-z)=T_2=C$ (see \eqref{eq:cat-neg} for $T_n$ with $n<0$). Similarly, letting $n=1$ yields $f=zC(z^3)$ and $B_f=z$; and letting $n=2$ yields $f=zC(z^5C(z^5))$ and $B_f=z^2T_3(-z^5)$.

In the second family, letting $n=2$ yields $f=z/C(-zC^2)$ and $B_f=T_4^2$ (\seqnum{A069271}). Likewise, letting $n=3$ yields $f=z/C(-z^3T_3^3(z^3))$ and $B_f=zT_6^3(z^3)$ (\seqnum{A212072} with $0$ prepended).
\end{example}

\begin{example} \label{ex:hfh_bh_123}
We have $P_0=1$, $P_1=3+z$, $P_2=5+5z+z^2$, which yields, for the first family,
\[
\begin{split}
B_{z\tilde{m}}&=B_{h_{1,0}}=T_{-1}^{-1}(-z)=C,\\
B_{zC^3}&=B_{h_{1,1}}=P_1(z)=3+z,\\
B_{zC^5(zC)}&=B_{h_{1,2}}=T_3(-z)P_2(zT_3^2(-z))\\
&=T_3(-z)(5+5zT_3^2(-z)+(zT_3^2(-z))^2)\\
&=(5T_3-5zT_3^3+z^2T_3^5)\circ(-z)=5+z^2T_3^5(-z)\\
&=5+\sum_{n=2}^{\infty}{(-1)^n\frac{5}{2n+1}\binom{3n-2}{n-2}z^n}\\
&=5+z^2-5z^3+25z^4-130z^5+\dots,
\end{split}
\]
and, similarly, for the second family,
\[
\begin{split}
B_{z/C(-zC^2)}&=B_{h_{2,2}}=T_4^2\\
&=\sum_{n=0}^{\infty}{\frac{2}{3n+2}\binom{4n+1}{n}z^n}=1+2z+9z^2+52z^3+340z^4+\cdots,\\
B_{z/C^3(-zT_3^3)}&=B_{h_{2,3}} = T_6^3\, P_1(zT_6^6) =T_6^3(3+zT_6^6)=T_6^3(2+T_6)=2T_6^3+T_6^{4}\\
&=\sum_{n=0}^{\infty}{\frac{3}{2n+1}\binom{6n+4}{n}z^n}=3+10z+72z^2+660z^3+6825z^4+\cdots.
\end{split}
\]
\end{example}

\begin{example} \label{ex:m-gen}
The form of the functions $\gamma$ in \eqref{eq:hfh-ctpos_gamma} and \eqref{eq:hfh-ctneg_gamma} corresponding to Examples \ref{ex:hfh-ctpos} and \ref{ex:hfh-ctneg}, respectively, leads us to the following generalization containing both of the above families:
\begin{equation}
\gamma(z) = z^k(1-z+z^2)^n.
\end{equation} 
Indeed, suppose the function a function $g=g(z)$ satisfies the functional equation
\[
g=1+z\gamma(g)=1+zg^k(1-g+g^2)^n.
\]
Note that $\gamma(z)$ is palindromic with $\dar(\gamma)=(k+k)+(0+2)\cdot n=2k+2n$. Therefore, the function $h=h(z)=zg^{2k+2n-1}$ is pseudo-involutory by Theorem \ref{thm:palindrome}, and thus, $f=f(z)=zg(z^{2k+2n-1})$ is also pseudo-involutory by Theorem \ref{thm:aeration}. Moreover,
\[
\delta^2\left(\frac{(z-1)^2}{z}\right)=\frac{\gamma(z)^2}{z^d}=\frac{z^{2k}(1-z+z^2)^{2n}}{z^{2n+2k}}=\left(\frac{1-z+z^2}{z}\right)^{2n}=\left(1+\frac{(z-1)^2}{z}\right)^{2n},
\]
so $\delta^2(z)=(1+z)^{2n}$ and $D=T_{2n}^n$\,, as in the proof of Theorem \ref{thm:hfh_bf_bh}. 

For negative indices $n=-p<0$, recall from \eqref{eq:cat-neg} that $T_{-n}(z)=1/T_{n+1}(-z)$, and therefore,
\[
D=T_{2n}^{n}=T_{-2p}^{-p}=\frac{1}{T_{2p+1}^{-p}(-z)}=T_{2p+1}^p(-z)=T_{-2n+1}^{-n}(-z).
\]

Note that the darga $d=2k+2n$ is even, so $d-1$ is odd. Therefore, both $f=zg(z^{2k+2n-1})$ and $h=zg^{2k+2n-1}$ are pseudo-involutory by Theorem \ref{thm:aeration}, and 
\begin{equation} \label{ex:m-gen-bf-bh}
B_f=B_f(z)=z^{k+n-1}T_{2n}^n(z^{2k+2n-1}) \qquad \text{and} \qquad B_h=B_h(z)=T_{2n}^n P_{k+n-1}(zT_{2n}^{2n})
\end{equation}
by Corollary \ref{cor:aeration_bf_d}.

\end{example}

\begin{example} \label{ex:parity-rna} 
\emph{(Modulo-$k$-matched noncrossing matchings)} 
Replacing $k$ with $k-1$ and letting $n=1$ in Example \ref{ex:m-gen}, we get
\[
g=1+zg^{k-1}(1-g+g^2), \qquad f=zg(z^{2k-1}), \qquad h=zg^{2k-1},
\]
and since $T_{2n}^n=T_2=C$, it follows that
\[
B_f=z^{k-1}C(z^{2k-1}), \qquad B_h=C P_{k-1}(zC^2).
\]
For $k=1$, we recover $g=\tilde{m}$, $f=h=z\tilde{m}$,  $B_f=B_h=C$, as in Examples \ref{ex:motzkin-tree} and \ref{ex:motzkin}.

For $k=2$, we find that $g$ is the generating function for the sequence \seqnum{A106228}, and 
\[
B_f=zC(z^3), \qquad B_h=C P_1(zC^2) = C(3+zC^2)=C(2+C)=C^2+2C,
\]
where the latter is the generating function for the sequence \seqnum{A038629}.

This example arises combinatorially in several different contexts (see \seqnum{A106228}), and we give another one here. Consider partial noncrossing matchings (see Example \ref{ex:motzkin-tree}) on $p$ points labeled $1$ through $p$ from left to right where only points with labels congruent modulo $k$ can be matched. For $k>1$, no pair of points with consecutive labels can be matched, so these noncrossing matchings are also secondary RNA structures of Example \ref{ex:rna}.

Let $a_p$ be the number of such structures on $p$ points, and consider the functions
\[
G_i=G_i(z)=\sum_{j=0}^{\infty}{a_{kj+i}z^j}, \qquad i=0,1,\dots,k-1.
\]
The point $1$ is either unmatched or matched to some point $k\ell+1$, where $\ell\ge 0$ and $2\ell+1<p$. Matching the point $1$ to point $k\ell+1$ separates the noncrossing matching into two submatchings: on the $k\ell-1$ points labeled $2$ through $k\ell$, and on the $p-k\ell-1$ points labeled $k\ell+2$ through $p$. The former submatching is necessarily of length $k-1\negthickspace\mod k$, whereas $p-k\ell-1\equiv p-1\negthickspace\mod k$. This yields the following generating function equations:
\[
\begin{split}
G_0(z^k)&=1+z(z^{k-1}G_{k-1}(z^k))+z^2(z^{k-1}G_{k-1}(z^k))(z^{k-1}G_{k-1}(z^k)),\\
z^iG_i(z^k)&=\phantom{\, 1+\ } z(z^{i-1}G_{i-1}(z^k))+z^2(z^{k-1}G_{k-1}(z^k))(z^{i-1}G_{i-1}(z^k)), \quad i=1,2,\dots,k-1,
\end{split}
\]
where the summands from left to right correspond to the empty string, point $1$ unmatched, and point $1$ matched, respectively. Therefore,
\[
\begin{split}
G_0&=1+zG_{k-1}+(zG_{k-1})^2,\\
G_i&=G_{i-1}+zG_{k-1}G_{i-1}=(1+zG_{k-1})G_{i-1}, \quad i=1,2,\dots,k-1,
\end{split}
\]
which implies that 
\[
G_{k-1}=(1+zG_{k-1})^{k-1}\left(1+zG_{k-1}+(zG_{k-1})^2\right).
\]
Now let $g(z)=1+zG_{k-1}(z)$. Then $g$ satisfies the functional equation $g=1+zg^{k-1}(1-g+g^2)$ and corresponds to the counting sequence for the noncrossing matchings of length $k\ell$  ($\ell\ge 0$), where the point $1$ is unmatched if the structure is nonempty, and only points with labels congruent modulo $k$ can be matched.

As noted above, for $k=2$, $g(z)=1+zG_1(z)$ is the generating function for the sequence \seqnum{A106228}. We also note that $G_0(z)$ is the generating function for \seqnum{A109081}, whereas $G_0(z^2)+zG_1(z^2)=\sum_{j=0}^{\infty}{a_{j}z^j}$ is the generating function for \seqnum{A215067}.

\end{example}

\begin{example} \label{ex:hfh-vk}
The identity \eqref{eq:cat-id} is one of many that can be used to produce pseudo-involutions in $k$-Bell subgroups involving compositions of functions. Here we find a two-parameter family of pseudo-involutions similar to those in Examples \ref{ex:hfh-ctpos} and \ref{ex:hfh-ctneg} using an identity similar to \eqref{eq:cat-id}.

Consider a family of functions $v_k=v_k(z)$, $k\in\mathbb{Z}$, defined implicitly by the equation
\begin{equation} \label{eq:def-vk}
v_k=1+z(v_k+v_k^k).
\end{equation}

The function
\begin{equation} \label{eq:vk-sol}
v_k=\frac{1}{1-z}T_k\left(\frac{z}{(1-z)^k}\right)
\end{equation}
is the generating function, when $k\ge 1$, for the number of lattice paths from $(0,0)$ to $(kn,0)$ on or above the $x$-axis with steps $U=(1,k-1)$, $D=(1,-1)$, and $L=(k,0)$, a generalization of Schr\"{o}der paths. Furthermore, $v_k$ is closely related to functions in several other families considered earlier in this paper. For example, $v_0=\frac{1+z}{1-z}=r_1=t_1=u_1$\,, $v_1=\frac{1}{1-2z}$\,, $v_2=r=r_2$ (see Section \ref{subsec:gen-schroeder}). Moreover, rewriting \eqref{eq:def-vk} as $v_k=1+zv_k(1+v_k^{k-1})$, we obtain $v_k=u_{k-1}(zv_k)$, i.e., $u_{k-1}=A_{zv_k}$ and $zv_k=\overline{(z/u_{k-1})}$, where $u_k$ is a function from the family in Example \ref{ex:double-leaf}. For nonpositive index values, i.e., for $v_{-k}$, $k\ge 0$, dividing equation \eqref{eq:def-vk} through by $v_{-k}$ and rearranging terms, we obtain
\[
\frac{1}{v_{-k}}=1-z\left(1+\left(\frac{1}{v_{-k}}\right)^{k+1}\right),
\]
which implies that
\[
v_{-k}(z)=\frac{1}{u_{k+1}(-z)}.
\]
It is easy to see that, for the function $v_k$, we have $\gamma=z+z^k$, a palindrome of darga $k+1$, so $(v_k,zv^k)$ is a pseudo-involution. Moreover, similarly to the identity \eqref{eq:cat-id}, we have
\begin{equation} \label{eq:vk-id}
1+zv_k^k=(1-z)v_k\,.
\end{equation}

As in Example \ref{ex:hfh-ctpos}, let $g=v_k$ and $h=zT_n^{n-1}$, then $g\circ h=v_k(zT_n^{n-1})$ and $f=zv_k^k$\,, so
\[
\widehat{h}\circ f\circ h=(z(1+z)^{n-1})\circ(zv_k^k)\circ(zT_n^{n-1}).
\]
However, from Equation \eqref{eq:vk-id},
\[
(z(1+z)^{n-1})\circ(zv_k^k)=zv^k\left(1+zv_k^k\right)^{n-1}=zv_k^k\left((1-z)v_k\right)^{n-1}=z(1-z)^{n-1}v_k^{n+k-1},
\]
so, recalling that $z(1-z)^{n-1}=\overline{zT_n^{n-1}}$, we have
\[
\widehat{h}\circ f\circ h=\left(z(1-z)^{n-1}v_k^{n+k-1}\right)\circ\left(zT_n^{n-1}\right)
=zv_k^{n+k-1}(zT_n^{n-1}).
\]
Thus,
\begin{equation} \label{eq:vk-ctpos}
\left(v_k(zT_n^{n-1}),zv_k^{n+k-1}(zT_n^{n-1})\right)
\end{equation} 
is a pseudo-involution in the $(n+k-1)$-Bell subgroup.

Similarly, as in Example \ref{ex:hfh-ctneg}, let $g=1/v_k$ and $h=-zT_n^n$, then $g\circ h=1/v_k(-zT_n^n)$ and $f=zv_k^k$\,, so
\[
\widehat{h}\circ f\circ h=(-z/(1+z)^n)\circ(zv_k^k)\circ(-zT_n^n).
\] 
However, from Equation \eqref{eq:vk-id},
\[
\left(-\frac{z}{(1+z)^n}\right)\circ(zv_k^k)=-\frac{zv_k^k}{(1+zv_k^k)^n}=-\frac{zv_k^k}{((1-z)v_k)^n}=-\frac{z}{(1-z)^n v_k^{n-k}}\,,
\]
so, recalling that $-z/(1-z)^n=\overline{(-zT_n^n)}$, we have
\[
\widehat{h}\circ f\circ h=\left(-\frac{z}{(1-z)^n v_k^{n-k}}\right)\circ(-zT_n^n)=\frac{z}{v_k^{n-k}(-zT_n^n)}.
\]
Thus,
\begin{equation} \label{eq:vk-ctneg}
\left(\frac{1}{v_k(-zT_n^n)},\frac{z}{v_k^{n-k}(-zT_n^n)}\right)
\end{equation}
is a pseudo-involution in the $(n-k)$-Bell subgroup.

Furthermore, since $1/(1-zT_n^{n-1})=1-(-zT_n^n)=T_n$, we have
\[
\begin{split}
\left(v_k(zT_n^{n-1}),zv_k^{n+k-1}(zT_n^{n-1})\right)&=\left(T_n\,T_k(zT_n^{n+k-1})\,,\, zT_n^{n+k-1} T_k^{n+k-1}(zT_n^{n+k-1})\right)\\
&=\left(T_n\,T_k(zT_n^{n+k-1})\,,\, (zT_k^{n+k-1})\circ(zT_n^{n+k-1})\right)
\end{split}
\]
and
\[
\begin{split}
\left(\frac{1}{v_k(-zT_n^n)},\frac{z}{v_k^{n-k}(-zT_n^n)}\right)
&=\left(\frac{T_n}{T_k(-zT_n^{n-k})}\,,\,\frac{zT_n^{n-k}}{T_k^{n-k}(-zT_n^{n-k})}\right)\\
&=\left(\frac{T_n}{T_k(-zT_n^{n-k})}\,,\,(-zT_k^{k-n})\circ(-zT_n^{n-k})\right).
\end{split}
\]
Note also that, for the first components of these pseudo-involutions, we have
\[
\begin{split}
T_n\,T_k(zT_n^{n+k-1})&=\frac{T_k(z)}{T_k(-z)}\circ(zT_n^{n+k-1}),\\
\frac{T_n}{T_k(-zT_n^{n-k})}&=\frac{T_k(z)}{T_k(-z)}\circ(zT_n^{n-k}).
\end{split}
\]

We leave it as an exercise for the reader to produce the $B$-functions for the pseudo-involutions \eqref{eq:vk-ctpos} and \eqref{eq:vk-ctneg} following the example of Theorem \ref{thm:hfh_bf_bh}. A natural combinatorial interpretation of these pseudo-involutions would also be of interest.
\end{example}

\begin{example} \label{ex:hfh-wk}
Finally, Examples \ref{ex:hfh-ctpos}, \ref{ex:hfh-ctneg}, and \ref{ex:hfh-vk} are all part of a three-parameter family of involutions in $p$-Bell subgroups (for various powers $p$) involving composition with either $zT_n^{n-1}$ or $-zT_n^n$. Given $d,k,n\in\mathbb{Z}$, $k>0$, consider a power series $w_k$ with $w_k(0)=1$ that satisfies an equation similar to \eqref{eq:cat-id} and \eqref{eq:vk-id}:
\begin{equation} \label{eq:wk-id}
1+zw_k^{d-1}=(1-z)w_k^k.
\end{equation}
Solving for $z$, we obtain
\[
z=\frac{w_k^k-1}{w_k^{d-1}+w_k^k}=\frac{(w_k-1)\sum_{j=0}^{k-1}{w_k^j}}{w_k^{d-1}+w_k^k},
\]
or, equivalently,
\[
w_k=1+z\frac{w_k^{d-1}+w_k^k}{\sum_{j=0}^{k-1}{w_k^j}}.
\]
Since the function $\gamma(z)=(z^{d-1}+z^k)/(\sum_{j=0}^{k-1}{z^j})$ is a generalized palindrome (having a palindromic numerator and denominator) of darga 
$\dar(\gamma)=(d-1+k)-(0+k-1)=d$, Theorem \ref{thm:palindrome} implies that $(w_k,zw_k^{d-1})$ is a pseudo-involution.

Now, as in Examples \ref{ex:hfh-ctpos}, \ref{ex:hfh-ctneg}, and \ref{ex:hfh-vk}, applying Theorem \ref{thm:inverse-prop} to $(g,f)=(w_k,zw_k^{d-1})$ and $h=zT_n^{n-1}$ yields the pseudo-involution
\begin{equation} \label{eq:wk-ctpos}
\left(w_k(zT_n^{n-1}),zw_k^{k(n-1)+(d-1)}(zT_n^{n-1})\right),
\end{equation}
while applying Theorem \ref{thm:inverse-prop} to $(g,f)=(1/w_k,zw_k^{d-1})$ and $h=-zT_n^n$ yields the pseudo-involution
\begin{equation} \label{eq:wk-ctneg}
\left(\frac{1}{w_k(-zT_n^n)},\frac{z}{w_k^{kn-(d-1)}(-zT_n^n)}\right).
\end{equation}
For example, letting $k=2$ and $d=4$ yields Examples \ref{ex:hfh-ctpos} and \ref{ex:hfh-ctneg}, whereas letting $k=1$ yields Example \ref{ex:hfh-vk} for $v_{d-1}$. As an exercise, the reader is invited to make the necessary minor adjustments needed to adapt  formulas \eqref{eq:wk-ctpos} and \eqref{eq:wk-ctneg} to the case when $k<0$.
\end{example}

\subsection{Pseudo-conjugation of exponential pseudo-involutions} \label{subsec:b-pal-exp}

All the pseudo-conjugation results stated in Section \ref{sec:b-pal} for ordinary Riordan arrays $(g,f)$ also apply to exponential Riordan arrays $[g,f]$. In this section we give several examples of pseudo-conjugation of exponential pseudo-involutions. In all of these examples, we start with the exponential pseudo-involution $[e^z,z]$, the Pascal triangle.

\begin{example} \label{ex:involutions} \emph{(Involutions)}
Let $[g,f]=[e^z,z]$ and $h=z+z^2/2$. Then $g\circ h=e^{z+z^2/2}$ is the exponential generating function for the number \seqnum{A000085} of involutions on an $n$-element set. Moreover, $h=z+z^2/2=(2z)\circ(z+z^2)\circ(z/2)$, so $\overline{h}=(2z)\circ\overline{(z+z^2)}\circ(z/2)=(2z)\circ(zC(-z))\circ(z/2)=zC(-z/2)$, and thus
$\widehat{h}=(-z)\circ(zC(-z/2))\circ(-z)=zC(z/2)=1-\sqrt{1-2z}$. Therefore,
\[
\widehat{h}\circ f\circ h = \left(zC\left(\frac{z}{2}\right)\right) \circ \left(z+\frac{z^2}{2}\right) = 1-\sqrt{1-2z-z^2},
\]
the exponential generating function of \seqnum{A182037}, and thus $\left[e^{z+z^2/2}, 1-\sqrt{1-2z-z^2}\right]$ is a pseudo-involution.
\end{example}

\begin{example} \label{ex:bell} \emph{(Set partitions and necklaces of partitions)}
Let $[g,f]=[e^z,z]$ and $h=e^z-1$. Then $g\circ h=e^{e^z-1}$ is the exponential generating function for the Bell numbers \seqnum{A000110}. We also have $\widehat{h}=(-z)\circ\overline{h}\circ(-z)=(-z)\circ\log(1+z)\circ(-z)=\log(1/(1-z))$. Therefore, the exponential pseudo-involutory companion of the Bell numbers is
\[
\widehat{h}\circ f\circ h = \log\left(\frac{1}{1-z}\right) \circ z \circ (e^z-1)
=\log\left(\frac{1}{2-e^z}\right),
\]
the exponential generating function of the sequence \seqnum{A000629} prepended by $0$, which counts necklaces of set partitions. This yields the exponential pseudo-involution
\[
\left[e^{e^z-1},\log\left(\frac{1}{2-e^z}\right)\right]
\]
that begins
\[
\begin{bmatrix}
1 & 0 & 0 & 0 & 0 & 0\\ 
1 & 1 & 0 & 0 & 0 & 0\\ 
2 & 4 & 1 & 0 & 0 & 0\\ 
5 & 18 & 9 & 1 & 0 & 0\\ 
15 & 94 & 72 & 16 & 1 & 0\\
52 & 575 & 600 & 200 & 25 & 1\\
\end{bmatrix}.
\]
\end{example}

\begin{example} \label{ex:bell-marked} \emph{(Set partitions with marked elements and rooted labeled trees)}
Let $[g,f]=[e^z,z]$ and $h=ze^z$. Then $g\circ h=e^{ze^z}$ is the exponential generating function for \seqnum{A000248} that counts set partitions with one marked element per part. Then
\[
\widehat{h}\circ f\circ h = (-z)\circ\overline{(ze^z)}\circ(-z)\circ(ze^z) = zS
\]
by \eqref{eq:labeled-trees-subtrees-pseudo}, where $S$ and $zS$ are from Example~\ref{ex:labeled-trees-subtrees}, and $zS$ is the exponential generating function of \seqnum{A216857}, i.e., the coefficients of $zS$ count rooted trees labeled with parts of a partition with one marked element per part. This yields the exponential pseudo-involution $\left[e^{ze^z},zS\right]$
that begins
\[
\begin{bmatrix}
1 & 0 & 0 & 0 & 0\\ 
1 & 1 & 0 & 0 & 0\\ 
3 & 6 & 1 & 0 & 0\\ 
10 & 45 & 15 & 1 & 0\\ 
41 & 432 & 210 & 28 & 1\\
\end{bmatrix}.
\]
\end{example}

\section{Conclusion} \label{sec:conclusion}

In this paper, we have developed some general methods for finding ordinary and exponential pseudo-involutions and their $B$-sequences, including, in particular, those in a $k$-Bell subgroup for some $k$. We also found several new families of pseudo-involutions of combinatorial significance. Many of the generating function identities involved, however, still await their combinatorial interpretation. Those, in particular, include combinatorial interpretations of several new integer sequences in our paper, as well as of the fact that the pseudo-involutory functions we have considered are indeed pseudo-involutory. We intend to develop such interpretations in subsequent papers and will only note here that many of those stem from the generalizations of the Carlitz-Scoville-Vaughn Theorem \cite{CSV} given by Parker \cite{Parker} for ordinary generating functions and by Drake \cite{Drake} for exponential generating functions.

\bigskip
\hrule
\bigskip

\noindent 2020 \emph{Mathematics Subject Classification}: Primary 05A15. Secondary 05A05, 05A19, 11B83.

\noindent \emph{Keywords:} Riordan group, Riordan array, pseudo-involution.

\bigskip
\hrule
\bigskip

\noindent (Concerned with sequences 
\seqnum{A000085},
\seqnum{A000108},
\seqnum{A000110},
\seqnum{A000245},
\seqnum{A000248},
\seqnum{A000272},
\seqnum{A000629},
\seqnum{A001003},
\seqnum{A001006},
\seqnum{A002212},
\seqnum{A004148},
\seqnum{A006318},
\seqnum{A007106},
\seqnum{A007559},
\seqnum{A025227},
\seqnum{A027307},
\seqnum{A038049},
\seqnum{A038629},
\seqnum{A049310},
\seqnum{A068875},
\seqnum{A069271},
\seqnum{A086246},
\seqnum{A089946},
\seqnum{A106228},
\seqnum{A109081},
\seqnum{A111125},
\seqnum{A127632},
\seqnum{A153295},
\seqnum{A153396},
\seqnum{A156308},
\seqnum{A166135},
\seqnum{A182037},
\seqnum{A200031},
\seqnum{A212072},
\seqnum{A215067},
\seqnum{A216857},
\seqnum{A238113},
\seqnum{A344623},
\seqnum{A347953},
\seqnum{A348189},
\seqnum{A348197}, and
\seqnum{A349562}.)

\bigskip
\hrule
\bigskip

\vspace*{+.1in}
\noindent
Received December 29 2021;
revised version received March 14 2022.
Published in \emph{Journal of Integer Sequences}, March 24 2022.

\bigskip
\hrule
\bigskip

\noindent
Return to
\htmladdnormallink{Journal of Integer Sequences home page}{http://www.cs.uwaterloo.ca/journals/JIS/}.
\vskip .1in

\end{document}